\documentclass[11pt]{amsart}
\usepackage[a4paper,margin=28mm]{geometry}
\usepackage{amsmath,amssymb}
\usepackage[T1]{fontenc}
\usepackage{lmodern,microtype}
\usepackage{xcolor}
\usepackage[colorlinks=true,linkcolor=blue!45!black,citecolor=blue!45!black,urlcolor=blue!45!black]{hyperref}
\allowdisplaybreaks[1]
\numberwithin{equation}{section}
\newtheorem{theorem}{Theorem}[section]
\newtheorem{proposition}[theorem]{Proposition}
\newtheorem{lemma}[theorem]{Lemma}

\theoremstyle{remark}

\DeclareMathOperator{\ind}{ind}
\DeclareMathOperator{\diver}{div}

\DeclareMathOperator{\coker}{coker}
\DeclareMathOperator{\diag}{diag}

\begin{document}

\title
{Finite Morse index solutions of the Allen--Cahn equation in $\mathbb{R}^3$}

\author[Y. Liu]{Yong Liu}
\address{School of Mathematics and Statistics, Beijing Technology and Business University, Beijing, China}
\email{yliumath@btbu.edu.cn}

\author[K.L. Wang]{Kelei Wang}
\address{School of Mathematics and Statistics, Wuhan University, Wuhan, China}
\email{wangkelei@whu.edu.cn}

\author[J.C. Wei]{Juncheng Wei}
\address{Juncheng Wei, Department of Mathematics, Chinese University of Hong Kong, Shatin, New Territories, Hong Kong}
\email{wei@math.cuhk.edu.hk}

 \author[K. Wu]{Ke Wu}
\address{School of Mathematics, Yunnan Normal University, Kunming, China}
\email{kewu@ynnu.edu.cn}

\begin{abstract}
We show that every bounded entire solution of the Allen--Cahn equation in dimension
three with finite Morse index has quadratic energy growth and finitely
many ends. If its index $I$ is positive, we prove that the number of ends is at most $2I$.
In particular, every Morse index-one solution has exactly two ends and is
axially symmetric.

\end{abstract}
\maketitle
\section{Introduction}\label{intro}

We consider bounded entire solutions of the Allen--Cahn equation
in dimension three,
\begin{equation}\label{ac}
 \Delta u=W'(u)\quad\text{in } \mathbb{R}^3,
 \qquad W(s)=\frac{(1-s^2)^2}{4}.
\end{equation}
The Allen--Cahn equation models phase boundaries in materials
science \cite{AC79}. For a bounded domain $\Omega\subset \mathbb{R}^3$, its associated
energy is$$
 E(u,\Omega)=\int_\Omega
 \left(\frac12|\nabla u|^2+W(u)\right)\,dx.$$
The function $W$ has two global minima, at $-1$ and $1$, and is called a
double-well potential. The second variation of the energy
is the quadratic form
$$
 Q_u(\xi)=\int_{\mathbb{R}^3}
 \left(|\nabla\xi|^2+(3u^2-1)\xi^2\right)\,dx,
 \qquad \xi\in C_c^\infty(\mathbb{R}^3).$$
The Morse index of $u$ is the largest dimension of a subspace
of $C_c^\infty(\mathbb{R}^3)$ on which $Q_u$ is negative definite.
A solution is stable when its Morse index is zero.

Our results are motivated by the relation between the Allen--Cahn
equation and minimal surfaces. The work of Modica and Mortola \cite{MM77}, Modica
\cite{Modica87}, and Sternberg \cite{Sternberg88} relates the
rescaled energy to the area of the limiting interface as the
transition thickness tends to zero. In the corresponding
minimization problems, this leads to area minimization with the
prescribed boundary or volume constraints.
For volume-constrained problems, Sternberg and Zumbrun \cite{SZ98}
studied how stability restricts the connectedness of phase boundaries
in strictly convex domains. Stability also passes geometric information
to limits of diffuse interfaces. Under uniform energy bounds,
Tonegawa \cite{Tonegawa05} studied limits of stable critical points,
and Tonegawa and Wickramasekera \cite{TW12} proved regularity
of the limiting interfaces.
In dimension three, Chodosh and Mantoulidis \cite{CM20}
proved estimates for the second fundamental form and the separation
of stable transition layers and applied these estimates to
sequences with bounded energy and bounded index.

A related symmetry question concerns bounded solutions that are
strictly monotone in one direction. Such solutions are stable because
their positive directional derivative solves the linearized equation.
The De Giorgi conjecture asks whether they are one-dimensional
in dimensions up to eight. Ghoussoub and Gui \cite{GG98} proved
this conclusion in the plane, and Ambrosio and Cabr\'e \cite{AC00}
proved it in three dimensions.
Savin \cite{Savin09} extended the conclusion through dimension
eight for monotone solutions that approach the opposite wells
in that direction. The examples of del Pino, Kowalczyk,
and Wei \cite{DKW11}, built from the minimal graphs of Bombieri, De Giorgi, and Giusti, show that the monotone symmetry statement
fails in dimensions nine and higher.
The stable version asks whether bounded stable solutions are
one-dimensional in dimensions at most seven.
The stable version of the De Giorgi conjecture states that all stable, bounded solutions to the Allen-Cahn equation are one-dimensional up to seven. In dimension $n=4$, a important progress toward this conjecture was established by Florit-Simon and Serra \cite{FSS}, who proved the one-dimensional symmetry under the additional geometric assumption of a bounded energy density. In a recent manuscript \cite{Stable3}, Liu-Luo-Wang-Wei-Wei-Wu proved that every bounded stable entire solution of the Allen-Cahn equation in $\mathbb{R}^3$ is one-dimensional. As a consequence, the full De Giorgi conjecture in $\mathbb{R}^{4}$ is true. The same result was proved proved in \cite{CFFFS} independently. On 07 October 2026, the De Giorgi conjecture was claimed to be fully solved by Open AI.

Finite Morse index solutions form a natural class beyond stable solutions.
Their ends describe the separate parts of the zero set far from the
origin. In the plane, del Pino, Kowalczyk,
Pacard, and Wei \cite{DPKPW10} constructed solutions with
multiple ends, whose zero curves are asymptotic to half-lines.
Kowalczyk, Liu, and Pacard studied the space
of four-ended solutions \cite{KLP12} and classified that
family \cite{KLP13}. Modulo rigid motions and a choice of sign,
these solutions form a connected one-dimensional family containing
the saddle solution. Wang and Wei \cite{WW19} proved that
finite Morse index implies finitely many ends in the plane.
For the elliptic sine--Gordon equation, Liu and Wei \cite{LW2022}
classified all finite-index solutions and computed their indices.

Further examples arise from constructions near minimal surfaces.
On closed Riemannian manifolds, Pacard and Ritor\'e
\cite{PR03} constructed solutions near nondegenerate embedded minimal
hypersurfaces that separate the manifold into two regions.
In $\mathbb{R}^3$, del Pino, Kowalczyk, and Wei \cite{DPW13}
constructed solutions associated with complete, embedded,
nondegenerate minimal surfaces with at least two ends and
$\int_\Sigma|A|^2<\infty$,
where $A$ is the second fundamental form.
Their theorem applies in particular to catenoids and yields two-end
solutions of the Allen--Cahn equation.
Their construction assumes that the bounded solutions of the Jacobi equation on
the minimal surface arise only from rigid motions.
The constructed solutions have the same Morse index as the
corresponding minimal surfaces. Within the axially symmetric class,
Gui, Wang, and Wei \cite{GWW20} proved that an index-one
solution in three dimensions has two ends.

In \cite[Theorem 1.4]{FS26}, Florit-Simon proved that a finite-index
solution with quadratic energy growth has finitely many parallel
ends. Outside a compact set, its zero set consists of ordered graphs
over one plane, and their heights have logarithmic asymptotics.

Our main result in this paper is the following.
\begin{theorem}\label{main}
Let $u$ be a bounded entire solution of the Allen--Cahn equation
\eqref{ac} with positive finite Morse index $I$. Then, for some
constant $C$ depending on $u$,
\begin{equation}\label{hyp}
 E(u,B_R)\le C R^2\quad\text{for every }R\ge1.
\end{equation}
Moreover, the zero set has finitely many ends, and their number $N$ satisfies
\begin{equation}\label{finends}
 N\le2I.
\end{equation}
If $I=1$, then $u$ has exactly two ends and is axially symmetric.
More precisely, in this case, after an orthogonal change of coordinates
$x=(y,z)\in \mathbb{R}^2\times\mathbb R$, there exist
$a\in\mathbb{R}^2$ and a function $\widehat u$ such that$$
 u(y,z)=\widehat u(|y-a|,z).$$
\end{theorem}
 Note that
finite index makes the solution stable outside a ball.
The classification in \cite{Stable3} applies to the stable entire
limits of translates tending to infinity. We show that each complete component of the zero set sufficiently far from the
origin is an entire graph, and that their asymptotic planes are parallel.

 In order  to prove axial symmetry, one possible approach is to consider the
solution of the linearized equation generated by rotations.
We fix a possible horizontal center $a$ and set$$
 J=\left((y_1-a_1)\partial_{y_2}
       -(y_2-a_2)\partial_{y_1}\right)u,
 \qquad L_u=-\Delta+3u^2-1.$$
Rotational invariance implies $L_uJ=0$, so $J$ solves the linearized equation.
However, although $J$ vanishes on the vertical line $y=a$, its number of
nodal domains still depends on the global geometry of its zero set.
A second issue also appears, which concerns the sign of the quadratic form.
For a compactly supported cutoff, integration by parts yields
\begin{equation}\label{jacut}
 Q_u(\chi J)=\int_{\mathbb{R}^3}J^2|\nabla\chi|^2\,dx
 \qquad\text{for }\chi\in C_c^\infty(\mathbb{R}^3).
\end{equation}
The right-hand side is nonnegative. Restricting a solution of the linearized equation to
its sign regions and truncating at infinity therefore requires
an additional argument to produce the two negative directions
needed for an index-one contradiction.

Results in minimal-surface theory suggest first controlling the number
of ends and then proving rotational symmetry.
L\'opez and Ros \cite{LR89} classified complete, connected,
orientable minimal immersions of index one in $\mathbb{R}^3$
as the catenoid and Enneper's surface.
Schoen \cite{Schoen83} proved that a complete,
connected minimal hypersurface in $\mathbb{R}^{n+1}$ with
exactly two ends, regular at infinity, is a catenoid.
This reduces the classification to an estimate for the
number of ends when the required end regularity is available.
Chodosh and Gianocca \cite{CG26} followed this
approach to classify complete, connected, embedded minimal
hypersurfaces in $\mathbb{R}^{n+1}$ with
$\int_\Sigma|A|^n<\infty$ and index one as higher-dimensional catenoids.
In $\mathbb{R}^4$, this integrability assumption follows
from finite index for complete two-sided minimal immersions,
by Chodosh and Li \cite[Theorem 1.5]{CL24}.
They also proved that every complete, connected, two-sided
stable minimal immersion in $\mathbb{R}^4$ is a hyperplane
\cite[Theorem 1.1]{CL24}.

Harmonic functions provide a way to relate the end structure
to analytic estimates.
Li and Tam \cite{LT92} related spaces of harmonic functions
to the number and type of ends of complete manifolds.
Li and Wang \cite{LW01} used harmonic functions to study
the end structure under lower bounds on the Ricci tensor and spectral
assumptions.
For complete two-sided minimal surfaces in $\mathbb{R}^3$,
Ros \cite{Ros06} used harmonic one-forms to obtain
a lower bound for the Morse index in terms of the genus.
Chodosh and M\'aximo \cite{CM16,CM23} extended this approach
to include the number of ends.
These arguments associate scalar test functions with harmonic
one-forms and use second-variation identities to turn
dimension estimates for spaces of such forms into lower
bounds for the Morse index. Let $\Sigma$ be a complete connected two-sided minimal
surface immersed in $\mathbb{R}^3$, with finite Morse index, genus $g$,
and $N$ embedded ends. Chodosh and M\'aximo
\cite[Corollary 1.2]{CM23} proved$$
 3\ind(\Sigma)\ge 2g+4N-5.$$
This estimate inspires us for the corresponding estimates of the Allen--Cahn equation. The harmonic one-form arguments suggest looking for a first-order
equation whose solutions yield useful test functions for $Q_u$.
For the Allen--Cahn equation, we use pairs of complex-valued
functions on $\mathbb{R}^3$ in place of one-forms on a surface.
The classical algebra of the Pauli matrices \cite[Section 1]{BE01}
expresses the linearized operator as the product of a first-order
operator and its adjoint, minus a matrix-valued term.
For the positive-index solutions considered in Theorem~\ref{main},
this matrix is positive definite.
For solutions of the first-order equation, the second-variation
identity contains only the negative matrix term and a cutoff
error. When this error tends to zero, the resulting negative
sum of scalar quadratic forms allows us to bound the dimension
of the space of such solutions in terms of the Morse index.

In dimension three, the Pauli matrices act on functions with
two components. The quadratic energy bound controls the
$L^2$ growth of the explicit solutions of the first-order equation.
Logarithmic cutoffs then make the error in the second-variation
identity tend to zero. Clifford matrices yield analogous
factorizations in higher dimensions, with more components.
The natural energy bound of order $\mathbb{R}^{n-1}$ in
$\mathbb{R}^n$ for $n\ge4$, however, is insufficient by itself
to ensure that the cutoff error vanishes. Thus the algebraic
factorization extends beyond three dimensions, whereas an
extension of the index argument used here requires further estimates
for the solutions or a different treatment of the cutoff error.

The first-order factorization is also available in dimension two.
A corresponding analysis of the ends would lead to the dimension
bound $N\le4I$ for $I>0$. We restrict the proof here to dimension
three. For $I>1$, this bound is weaker than
Mantoulidis's estimate $N\le2I+2$ for multi-ended solutions
\cite[Theorem 1.5]{Man21}.
For the elliptic sine--Gordon equation, the classification in
\cite{LW2022} yields the exact relation $I=N(N-2)/8$, where $N$
is the total number of ends.

The paper is organized as follows. Section~\ref{facsec} proves quadratic energy growth and finiteness
of the ends.
Section~\ref{sec-f} develops the first-order decomposition of the
linearized operator. Section~\ref{ends} proves the bound $N\le2I$
and its consequence $N=2$ when $I=1$.
Section~\ref{symm} determines the two-end asymptotics needed for
the moving-plane argument and completes the symmetry proof.

\section{Quadratic energy growth and the geometry of the ends}
\label{facsec}

In this section, we  prove that finite Morse index implies quadratic energy growth. Moreover, we use it to investigate the geometric structure of the zero set.

Throughout this section, $u$ is a bounded entire solution of
\eqref{ac} with finite Morse index. We denote its zero set by
$\Sigma=\{u=0\}$.
It is easy to see that finite Morse index implies stability outside some ball $B_{R_0}$.

Our starting point is the next lemma, which shows that outside a compact subset, the zero set is a smooth hypersurface and
becomes flat relative to its distance from the origin.

\begin{lemma}\label{exterior}
For sufficiently large $R_0$, $\Sigma$ outside $B_{R_0}$ is a smooth
embedded surface, with $$|\nabla u|\ge c>0,\quad\text{on}~\Sigma\setminus B_{R_0}.$$
Moreover, we have
\begin{equation}\label{farflat}
 \lim_{R\to\infty}\sup_{x\in\Sigma\setminus B_R}|x||A(x)|=0,
 \qquad |x|^2|H(x)|\le C,
\end{equation}Here  for unit normal $\nu$,  we use the convention of the second fundmental form
$A(X,Y)=\langle D_X\nu,Y\rangle$ for tangent vectors $X,Y$.
The mean curvature is $H=\operatorname{tr}A$.
\end{lemma}

\begin{proof}
The proof of the first statement is similar to \cite[Lemma 3.2]{WW19}. By applying \cite[Corollary 1.3]{Stable3} on balls
whose radius is a fixed small multiple of $|x|$. We obtain
\begin{equation}\label{curvature estimates}
|A(x)|\le C/|x|,\qquad |H(x)|\le C/|x|^2,
\end{equation}
where $C$ is a universal positive constant. As discussed in \cite[Lemma 2.9, Theorem 2.10]{FS26}, this implies the local graph and
nondegeneracy assumptions needed in \cite[Theorem 1.1]{WW19b}.

In order to finish the proof, it suffices to show the first statement in \eqref{farflat}. Suppose that it fails. We can choose
$x_j\in\Sigma$ with $R_j=|x_j|\to\infty$ and
$$R_j|A(x_j)|\ge c>0.$$
Let
$$\Sigma_{j}:=R_j^{-1}\Sigma,\qquad \epsilon_{j}:=R_{j}^{-1}.$$
It follows from \eqref{curvature estimates} that the second fundamental forms of the surfaces $\Sigma_{j}$ are uniformly
bounded  on each compact subset of
$\mathbb{R}^3\setminus\{0\}$.  Moreover, the mean curvature $H_{\Sigma_{j}}$ tend to zero. Therefore, up to a subsequence,
the local graph estimates  yield a limit $\Sigma_{\infty}$ made
up of disjoint minimal surfaces, with local $C^2$ convergence of the
individual graphs.  Assume $x_j/R_j\to x_{\infty}$ as $j\to\infty$, the  local $C^2$ convergence yields $|A(x_{\infty})|\ge c$. Let $L$ be the leaf of $\Sigma_{\infty}$ through $x_{\infty}$.

\textbf{Claim 1}: the leaf $L$ have stable two-sided covers.

 Every point of $L$ is a limit of points of $\Sigma_j$, so on each compact piece of $L$ the
sheets of $\Sigma_j$ converge, up to a subsequence, in $C^2$ to sheets over
the piece. If two distinct sheets $f_j^{1}<f_j^{2}$ converge to the same
compact piece of $L$, their positive separation $h_j:=f_j^{2}-f_j^{1}$
satisfies, by \cite[Proposition 10.1]{WW19b} (whose sheeting assumptions hold
on the piece),
$$
 h_j\ge c\,\varepsilon_j|\log\varepsilon_j|
 \qquad\text{pointwise on the piece.}
$$
Each sheet is a component of $\{u_j=0\}$, so its mean curvature is
$O(\varepsilon_j)$ by \eqref{curvature estimates}. Writing $f_{j}^{1}, f_{j}^{2}$ as
normal graphs over $L$ and subtracting their mean curvature equations, the
linearization at $L$ is the Jacobi operator $\Delta+|A|^{2}$, and each term
of the remainder contains at least one factor of $g$ or $\nabla g$, where
$g$ is a graph function; since both sheets converge to $L$ in $C^{2}$, on
each compact leaf chart we obtain
$$
 |(\Delta+|A|^{2})h_j|
 \le o(1)\bigl(|\nabla^{2}h_j|+|\nabla h_j|+h_j\bigr)+C\varepsilon_j .
$$
Here the differential operators refer to the limiting leaf.  Since
$h_j\ge c\varepsilon_j|\log\varepsilon_j|$ pointwise,
$\varepsilon_j/h_j\le C/|\log\varepsilon_j|$, so the last term can be
absorbed into the zeroth order coefficient of $h_j$:
$$
 \bigl|(\Delta+|A|^{2}+b_j)h_j\bigr|
 \le o(1)\bigl(|\nabla^{2}h_j|+|\nabla h_j|\bigr),
 \qquad |b_j|\le\frac{C}{|\log\varepsilon_j|}\to0 .
$$
Fix $p$ on the chart and set $\varphi_j:=h_j/h_j(p)$. For $j$ large the
coefficient of $|\nabla^{2}\varphi_j|$ is less than $1/2$, so after
absorbing it into the uniformly elliptic principal part, interior
$W^{2,q}$ estimates give a $j$-independent $C^{1,\alpha}$ bound for
$\varphi_j$ on every compact subchart, and the Harnack inequality for
equations with bounded zeroth order term gives $\sup\varphi_j\le C
\inf\varphi_j$ there. A diagonal subsequence over an exhaustion of $L$ (or
of the covering through which the sheets converge) then converges in
$C^{1}_{\mathrm{loc}}$ to $\varphi\ge0$ with $\varphi(p)=1$, and $\varphi$
solves the Jacobi equation
$$
 (\Delta_L+|A_L|^{2})\varphi=0
$$
strongly. The elliptic regularity theories yields $\varphi$ is smooth. Moreover, $\varphi$ is positive by the strong maximum principle.
A positive solution of the Jacobi equation is equivalent to stability: by
the ground state transform,
$$
 \int\bigl(|\nabla\xi|^{2}-|A|^{2}\xi^{2}\bigr)
 =\int\varphi^{2}\bigl|\nabla(\xi/\varphi)\bigr|^{2}\ge0
 \qquad\text{for every }\xi\in C_c^{1},
$$
see \cite{FCS80}. Hence the two-sided cover of $L$ is stable.

If only one transition surface approaches the leaf, a single local
graph and exponential decay away from $\Sigma$ implies the
rescaled solution $u_j(x):=u(R_jx)$ satisfies
\begin{equation}\label{rescaledenergy}
 \mathcal E_j(u_j,K):=\int_K\left(\frac{\varepsilon_j}{2}|\nabla u_j|^{2}
 +\frac{W(u_j)}{\varepsilon_j}\right)\,dx
 =R_j^{-2}E(u,R_jK)\le C_K,
\end{equation}
with $C_K$ independent of $j$. By \eqref{rescaledenergy} and the stability of $u_j$ on compact subsets of
$\mathbb R^{3}\setminus\{0\}$, stability passes to the limit by
\cite{Tonegawa05}, so the leaf is a stable minimal surface; it is two-sided,
being an embedded surface in $\mathbb R^{3}\setminus\{0\}$.

\textbf{Claim 2}: the limiting surface $\Sigma_{\infty}$ is complete outside the origin.

Let $\gamma\colon[0,1)\to L$ be a path of finite length. If $\gamma$ had no
limit in $\mathbb R^{3}$, then $|\gamma(t)|\to\infty$ along a subsequence,
contradicting $\operatorname{length}(\gamma)\ge|\gamma(t)|-|\gamma(0)|$.
Hence $\gamma(t)\to q$ for some $q\in\mathbb R^{3}$. Suppose $q\neq0$. Since
$\Sigma_{\infty}$ is a lamination of $\mathbb R^{3}\setminus\{0\}$, there is
a product chart $U\subset\mathbb R^{3}\setminus\{0\}$ around $q$ in which
$\Sigma_{\infty}\cap U$ is a union of parallel disks. As $U$ is open, the
tail of $\gamma$ lies in $U$; as it is connected, it lies in a single disk
$D$ of $L\cap U$. This disk is closed in $U$, so $q\in D\subset L$. Thus
every divergent path of finite length in $L$ can only limit onto the origin,
that is, $L$ is complete outside the origin in the sense of
\cite[Definition~3.1]{MPR16}.

Indeed, a path of finite length staying away from the origin has a limit in a local
lamination chart and continues on the same surface.

By the stable-surface result of \cite[Lemma 3.3]{MPR16}, we know that the
second fundamental form  of $\Sigma_{\infty}$ equals zero. This contradicts $|A|\ge c$ at the
chosen point and proves \eqref{farflat}.
\end{proof}

We will need length bounds for the curves where the zero set meets
large spheres. The next lemma shows that an embedded circle which
is locally close to geodesic arcs is close to one whole great
circle, with length close to $2\pi$.
This conclusion will be used to control the lengths of spherical sections
of distant ends.

\begin{lemma}\label{circles}
Let $\gamma_j$ be smooth embedded circles in the unit sphere.
We parametrize each circle by arclength and assume that
\begin{equation}\label{curve equation}
\sup|\gamma_j''+\gamma_j|\to0.
\end{equation}
After rotations, the circles are graphs over one equator, with
heights tending to zero in $C^1$. Moreover, their lengths tend to $2\pi$.
\end{lemma}

\begin{proof}
By the assumptions, we have $|\gamma_{j}|^{2}=|\gamma_{j}'|^{2}=1$. Then $\langle\gamma_{j}''+\gamma_{j}, \gamma_{j}\rangle=0$ and
$$|\gamma_{j}''|^{2}=1+o(1).$$
Fenchel's theorem implies
$$
 2\pi\le\int_{\gamma_j}|\gamma_j''|
        =(1+o(1))\operatorname{length}(\gamma_j).
$$
Thus the lengths stay bounded away from zero. Up to a subsequence, we may assume $\gamma_{j}\to\gamma$ as $j\to\infty$. The assumption \eqref{curve equation} and the equation for the arclength parametrization show that local
limits are geodesic arcs. Since the original curves are closed and have uniformly
bounded acceleration, these arcs extend to complete great
circles. Two transverse limiting arcs would force
intersections of the approximating curves. Thus all limiting arcs
belong to the same great circle, and both positions and tangent
lines approach it uniformly. For large $j$, projection onto that
circle is a local diffeomorphism and hence a finite covering.
An embedded essential circle in an annulus has covering degree
one. The asserted graph representation and length convergence
follow.
\end{proof}

For the following geometric argument, we use $r=|x|$ on a surface
$M$ and denote its induced metric by $g_M$. Direct differentiation yields
\begin{equation}\label{radhess}
 \operatorname{Hess}_M\frac{|x|^2}{2}
   =g_M-(x\cdot\nu)A\ge(1-r|A|)g_M.
\end{equation}
By Lemma \ref{exterior}, if $r$ is large enough, then $1-r|A|$ is positive.
Along an outward radial path parametrized by $r$, we have
\begin{equation}\label{radnormal}
 \frac{d}{dr}(x\cdot\nu)=rA(T,T),\qquad
 |\nabla_Mr|^2=1-\frac{(x\cdot\nu)^2}{r^2},
 \qquad T=\frac{\nabla_Mr}{|\nabla_Mr|}.
\end{equation}
If $\gamma_r$ is an arclength parametrization of the corresponding
intersection curve dilated to the unit sphere, the decomposition
of its acceleration gives
\begin{equation}\label{sphereest}
 |\gamma_r''+\gamma_r|
 \le\frac{r|A|+|x\cdot\nu|/r}
            {\sqrt{1-(x\cdot\nu)^2/r^2}}.
\end{equation}

For a complete component far away from the origin, we need a graph
description that applies to the whole component.
The next lemma  shows that all
such components have parallel asymptotic planes.
This common direction will allow us to order the components.

\begin{lemma}\label{fargraphs}
For sufficiently large $R_1$, every complete component of $\Sigma$
disjoint from $B_{R_1}$ is an entire graph with slope at most $1/10$
over its asymptotic plane. These planes are parallel for all such
components.
\end{lemma}

\begin{proof}
Let $\Sigma_j$ be a sequence of complete components
such that the minimum distances $d_j$ from the origin tend to infinity.
We dilate them to $M_j=d_j^{-1}\Sigma_j$. Equation~\eqref{farflat}
implies
$$
 \min_{M_j}r=1,\qquad
 \varepsilon_j:=\sup_{M_j}r|A|\longrightarrow0,
 \qquad r^2|H|\le C/d_j.
$$
Because the surfaces are components of a closed
regular zero set, thay are proper. By \eqref{radhess}, $|x|^2/2$ is a  strictly
convex function on $M_j$ for large $j$. Hence every critical point is a
nondegenerate minimum. Connectedness and Morse theory imply that
there is exactly one minimum and that every regular sublevel set
is a disk.  Thus $M_j$ is a topological plane and each
spherical section above its minimum is one embedded circle.

We rotate the nearest point so that its unit normal is $e_3$.
Integrating  \eqref{radnormal} along each radial path gives
$$
 |x\cdot\nu|\le1+\varepsilon_j(r-1).
$$
This implies
$$\frac{|x\cdot\nu|}{r}\le\varepsilon_j+\frac1r.$$
For large $r$, we get from Lemma~\ref{circles} and \eqref{sphereest} that the
section close to a great circle. On $R/8<r<8R$, the same estimates
imply $|\nabla_{M_j}r|\ge1/2$ and
$$
 \left|\frac{d\nu}{dr}\right|\le\frac{2\varepsilon_j}{r}.
$$
Consequently, for every fixed $\delta>0$, we can find $R_2$ and $j_2$
such that, for all $R\ge R_2$ and $j\ge j_2$, there exists a unit vector
$e_{j,R}$ such that
$$
 M_j\cap(B_{8R}\setminus\overline{B_{R/8}})
       \subset\{|x\cdot e_{j,R}|\le\delta R\}.
$$
The choices of $R_2$ and $j_2$ are uniform in the subsequent radius
$R$.

We apply \cite[Proposition 3.2]{FS26} with a fixed sufficiently
small $\delta$.
Outside one fixed cylinder, $M_j$ is a graph over a plane with unit
normal $e_j$, and its height satisfies
$$
 |f_j(y)|+|y||\nabla f_j(y)|\le C|y|^{1/2}.
$$
The cylinder and $C$ are independent of $j$.

We also control the part inside a fixed ball $B_S$. A minimizing
intrinsic geodesic from the nearest point to $x\in M_j\cap B_S$
has length at most $CS$, by integrating \eqref{radhess} twice.
Since $|A|\le\varepsilon_j$ everywhere on $M_j$, we have
$$
 \sup_{M_j\cap B_S}|\nu-e_3|\le C\varepsilon_j S.
$$
The exterior graph estimate implies $|\nu-e_j|\le CS^{-1/2}$
outside $B_S$, for $S$ larger than a fixed constant. Comparison
on the boundary gives
$$
 \sup_{M_j}|\nu-e_j|\le C\varepsilon_j S+CS^{-1/2}.
$$
We first let $j\to\infty$ and then let $S\to\infty$.
Orthogonal projection onto $e_j^\perp$ is therefore a local
diffeomorphism with uniformly bounded inverse differential.
Completeness allows every finite-length path in the plane to be
lifted. The projection is a covering of the plane and a global diffeomorphism. Moreover, its slope
tends to zero.

Finally, recall that the height of each end grows sublinearly over its
asymptotic plane. Two different plane directions would make the
signed height of one component relative to the other take both
signs far away. A path between those points would force an
intersection. The components are disjoint, so their asymptotic
planes are parallel.
\end{proof}

Note that the graph description still allows infinitely many separate
components. We would like to excludes this possibility under finite index assumption.
This finiteness is essential for the global energy estimate, since
each additional component can contribute energy. We have the following

\begin{lemma}\label{finitecomp}
For sufficiently large $R_1$, the set
$\Sigma\setminus\overline{B_{R_1}}$ has finitely many connected
components.
\end{lemma}

\begin{proof}
We choose $R_1>R_0+1$ large enough so that
Lemma~\ref{fargraphs} is applicable. By Lemma \ref{exterior}, the intersection of $\Sigma$ with $\partial B_{R_1}$
is a compact one-dimensional manifold, hence a finite union of
circles. By the implicit function theorem, there are only finitely many exterior components have closures meeting this sphere.
Every other component is complete and misses $B_{R_1}$. Using the implicit function theorem once again, we know that there are
only finitely many components meet any fixed compact annulus.

If infinitely many components remain, Lemma~\ref{fargraphs} implies they can be represented  as entire graphs of slope at most $1/10$ over
one plane. By taking suitable rotation if necessary, we may assume this plane equals $\{z=0\}$.
At every vertical-axis intersection, the gradient lower bound and
the nonzero vertical normal component imply $u_z\ne0$.
These intersections are isolated, and compactness excludes an
accumulation at finite height. Their heights therefore form a
locally finite set, unbounded above or below.
After reversing $z$ if necessary, we
choose a graph $f_0$ with $f_0(0)>3R_1$. Its slope bound places
$B_{R_1}$ entirely below it. Any zero-set component above this
graph misses the ball and is another entire graph.
We therefore choose three consecutive graphs
$f_1<f_2<f_3$ above $f_0$. In the region
$$
 S=\{(y,z):f_1(y)<z<f_3(y)\},
$$
the middle graph is the only interior zero set, and $S$ stays
outside $B_{R_1}$.

By \cite[Proposition 7]{HLSWW21}, we get that
\begin{equation}\label{phasedec}
 \frac12|\nabla u(x)|^2+W(u(x))
       \le C e^{-c\operatorname{dist}(x,\Sigma)}.
\end{equation}
A segment from $x\in S$ to a zero outside $\overline S$ crosses
one of its boundary graphs. Hence distance to the full zero set
equals distance to the union of these three graphs.
For a graph of slope at most $1/10$, its distance from $(y,z)$
is bounded below by a fixed multiple of $|z-f_i(y)|$. Therefore, we conclude that
$$
 \frac12|\nabla u(y,z)|^2+W(u(y,z))
       \le C\sum_{i=1}^3 e^{-c|z-f_i(y)|}
       \quad\hbox{in }S.
$$
Integration in $z$ and then in $|y|<R$ yields
$$E(u,S\cap B_R)\le C R^2.$$

On these three graphs, $u_z$ is nonzero and has alternating
signs. We choose $w=u_z$ or $w=-u_z$ so that $w>0$ on the middle
graph and $w<0$ on the outer graphs. The nodal domain $\Omega$
containing the middle graph lies in $S$. We have $L_uw=0$ and
$$
 \int_{\Omega\cap B_R}w^2\,dx\le2E(u,S\cap B_R)\le C R^2.
$$
Stability outside large ball ensures the existence of  a positive solution $\phi$ of
$L_u\phi=0$ on $D_0=\mathbb{R}^3\setminus\overline{B_{R_0}}$.
We extend $w|_\Omega$ by zero to $\widetilde w$ on $D_0$.
The zero trace on the nodal boundary and the positive distance
from $B_{R_0}$ allow compactly supported $H^1$ tests. We have
$$
 \int_{D_0}\phi^2
      \left|\nabla\left(\frac{\chi\widetilde w}{\phi}\right)\right|^2
   =Q_u(\chi\widetilde w)
   =\int_\Omega w^2|\nabla\chi|^2.
$$
We take $\chi=1$ on $B_T$, decrease it linearly in $\log|x|$
on $T<|x|<T^2$, and set it equals to zero outside $B_{T^2}$.
The quadratic bound for the integral of $w^2$ yields
$$
 \int_\Omega w^2|\nabla\chi|^2\le\frac{C}{\log T}.
$$
As $T\to\infty$, the quotient $\widetilde w/\phi$ must be
constant on the connected set $D_0$. It is positive in $\Omega$
and zero on an open set below the lowest graph. This is a
contradiction.
\end{proof}

There are now only finitely many exterior components, but their
area and the energy of $u$ still require bounds.
The next proposition establishes quadratic energy growth.

\begin{proposition}\label{density}
Every bounded entire solution of finite Morse index satisfies
\eqref{hyp}.
\end{proposition}

\begin{proof}
We assume that $\Sigma$ is nonempty. By \eqref{farflat},
we can arrange $r|A|<1/2$ outside a fixed ball.
Due to \eqref{radhess}, we have  every exterior critical point of
$r$ a nondegenerate minimum. A component meeting the inner sphere
has none, since a path from a minimum to that boundary would
force a saddle. A complete component has exactly one minimum.
Lemma~\ref{finitecomp} therefore leaves only finitely many minima.
We choose a regular radius $R_*$ beyond all of them. Radial flow
then identifies each exterior component with a circle times
$(R_*,\infty)$. The flow exists on every finite interval of radii
by properness and the absence of critical points.

Integration of \eqref{radnormal} from the inner circle tells us that
$$
 \frac{|x\cdot\nu|}{r}
 \le\frac{R_*}{r}+\frac1r\int_{R_*}^r
       \sup_{\Sigma\setminus B_s}|x||A(x)|\,ds
 \longrightarrow0.
$$
The convergence follows from \eqref{farflat}.
Thus $|\nabla_\Sigma r|\to1$.
By \eqref{sphereest} and Lemma~\ref{circles}, every spherical
section of an end has length at most $Cr$ for large $r$.
There are finitely many ends, so we have
$$
 \operatorname{Area}\bigl(\Sigma\cap(B_R\setminus B_{R_*})\bigr)
 =\int_{R_*}^R\int_{\Sigma\cap\partial B_r}
       \frac{1}{|\nabla_\Sigma r|}\,d\mathcal H^1\,dr
 \le C_u R^2.
$$

We choose a fixed small $\rho>0$ and a maximal $\rho$-separated
set of distant zeros $p_i$. Their $\rho$-balls cover the exterior
zero set, while their $\rho/3$-balls are disjoint.
The uniform local graph estimates provide a fixed positive
lower bound for the area in each smaller ball. The area bound
therefore gives
$$
 \#\{i:|p_i|\le2R+\rho\}\le C_u R^2.
$$
We include the remaining zero set in a fixed ball $K$.
For $x\in B_R$, either its distance from $\Sigma$ is at least
$R$, or a nearest zero belongs to $K$ or to one of the balls
just counted. Equation~\eqref{phasedec} consequently gives
$$
 \frac12|\nabla u(x)|^2+W(u(x))
 \le Ce^{-cR}+Ce^{-c\operatorname{dist}(x,K)}
       +C\sum_{|p_i|\le2R+\rho}e^{-c|x-p_i|}.
$$
The last exponentials have uniformly bounded integrals on
$\mathbb{R}^3$, and the integral involving $K$ is finite.
Integration over $B_R$ tells us that
$$E(u,B_R)\le C_u R^2.$$
This finishes the proof.
\end{proof}

For the rest of this section, $u$ satisfies the hypotheses of
Theorem~\ref{main}, with arbitrary positive finite Morse index $I$. We use the one-dimensional solution
$g(t)=\tanh(t/\sqrt2)$. Its energy is$$
 \sigma_0=\int_{\mathbb R}(g')^2\,dt=\frac{2\sqrt2}{3}.$$
From this point onward, $r$ denotes the horizontal radius. We set
$$x=(y,z),\qquad r=|y|,\qquad\langle r\rangle=(1+r^2)^{1/2}.$$
The notation $O_{C_b^k}(R^{-q})$ means that the same bound holds
after at most $k$ applications of $r\partial_r$ and
$\partial_\theta$.  When complex horizontal coordinates are useful,
we set$$
 w=y_1+iy_2,\qquad
 \partial_w=\frac12(\partial_{y_1}-i\partial_{y_2}),\qquad
 \partial_{\bar w}=\frac12(\partial_{y_1}+i\partial_{y_2}).$$
Hermitian inner products are linear in their first argument.
Unless specified otherwise, a norm on a vertical line is taken over
the entire line. We distinguish horizontal derivatives at fixed $z$
from derivatives at fixed $t=z-f_i(y)$ whenever both occur.
Constants may depend on the solution and its finitely many ends,
but are independent of the radius tending to infinity.

With \eqref{hyp} established, we can use the
structure theorem  in \cite[Theorem 1.4]{FS26} to obatin the following result.
\begin{theorem}\label{endstruct}
Let $u:\mathbb{R}^{3}\to[-1,1]$ be a finite-index solution of the Allen--Cahn equation satisfying \eqref{hyp}. Then $u$ has finitely many parallel ends, which are either planar
or catenoidal. More precisely, there exist some $R_{0}>0$ and $N\in\mathbb{N}$, as well as constants $b_{i}, c_{i}$ for every $i=1, 2,\cdots, N$, such that after an orthogonal change of coordinates, the zero set outside a
compact set is a finite family of ordered graphs
$$\{u=0\}=\bigcup_{i=1}^{N}\text{graph}\xi_{i},$$
where $\xi_{1}<\cdots<\xi_{N}$,
$$\xi_{i}=b_{i}+c_{i}\log|y|+O(|y|^{-\alpha})$$
for some $\alpha>0$ and $|c_{i+1}-c_{i}|>\sqrt{2}$ for every $i=1, 2,\cdots, N-1$.
\end{theorem}
 This theorem provides one common plane for all ends and determines their
logarithmic heights. The strict gap between consecutive growth
rates will control the interactions between ends in
Section~\ref{ends}.

We proceed to prove the following.

\begin{proposition}\label{geom}
After an orthogonal change of coordinates, the zero set outside a
compact set is a finite family of ordered graphs
$z=\zeta_i(y)$, $1\le i\le N$, over $\{r>R_0\}$.
There are smooth heights $f_i$ with the following properties.
For constants $c_i,b_i$,
\begin{equation}\label{logends}
\begin{aligned}
 &f_i=c_i\log r+b_i+o_{C_b^2}(1),\\
 &f_i-\zeta_i=O(r^{-2}),\\
 &|\nabla_y(f_i-\zeta_i)|+|D_y^2(f_i-\zeta_i)|=o(r^{-2}).
 \end{aligned}
\end{equation}
Moreover, $D_i=f_{i+1}-f_i$ tend to infinity, and
\begin{equation}\label{pi}
 p_i=\sqrt2(c_{i+1}-c_i)>2
 \end{equation}

\end{proposition}
\begin{proof}

By \cite[Section 1]{HLSWW21}, we know that $-1<u<1$.  According to
Theorem~\ref{endstruct}, the zero-graph heights can be denoted by
$\zeta_i$. Their ordering yields \eqref{pi}. The Toda equation in
\cite[Proposition 4.5]{FS26} and interior estimates on annuli upgrade
the logarithmic remainder to $o_{C_b^2}(1)$.
Moreover, the normal-coordinate correction and shifts are
$O(r^{-2})$ in $C^2$, with tangential derivatives $o(r^{-2})$ in
$C^1$, by \cite[Lemma 4.9 and equations (62)--(63)]{FS26}.
By \cite[Proposition 7]{HLSWW21}, we get that
\begin{equation}\label{phasedecay}
 \frac12|\nabla u(x)|^2+W(u(x))
       \le C e^{-c\operatorname{dist}(x,\Sigma)}.
\end{equation}
The logarithmic growth of the heights therefore makes the
tails outside $|z|\le L\log r$ smaller than any prescribed inverse
power of $r$ when $L$ is sufficiently large.

Let $q_i=g$ on an
increasing end and $q_i=-g$ on a decreasing end, and set
\begin{equation}\label{Uf}
 U_f(y,z)=q_1(-\infty)+\sum_{i=1}^N
       \bigl(q_i(z-f_i(y))-q_i(-\infty)\bigr).
\end{equation}
Put $\psi=u-U_f$ and $\gamma_i=q_i'(z-f_i(y))$. Since
$$\frac{\partial}{f_{j}}\langle u-U_{f}, \gamma_{i}\rangle=\langle\gamma_{i}, \gamma_{j}\rangle+o(1),$$
we can use the implicit function to choose $f_{i}, 1\leq i\leq N$ so that
$\langle\psi,\gamma_i\rangle_{L^2(\mathbb R_z)}=0$ for every $i$.

At the unadjusted heights $\zeta_i$, these orthogonality errors are
$O(r^{-2})$, and their first two horizontal derivatives are
$o(r^{-2})$. Indeed, the leading change from normal to vertical
coordinates has zero projection onto $q_i'$ by parity, and the
remaining terms satisfy the cited correction and derivative bounds.
The Jacobian with respect to the height shifts is
$\sigma_0\delta_{ij}+o(1)$, uniformly for shifts of size $O(r^{-2})$,
by separation and the exponential decay of the transition derivatives.
The implicit function argument used in \cite[Lemma 4.8]{FS26},
applied here to vertical lines, therefore selects the heights $f_i$.
Differentiating the orthogonality conditions twice gives
\eqref{logends}. The same approximation and phase-decay estimates
also imply $\|\psi(y,\cdot)\|_{L^\infty(\mathbb R_z)}\to0$ as
$r\to\infty$.
\end{proof}
 \section{A factorization of the linearized Allen-Cahn operator}\label{sec-f}
In this section, we  turn to the decomposition of the linearized
operator $L_{u}$. Recall that if $n=1$, then $g(t)=\tanh(t/\sqrt2)$ is the one dimensional solution. The linearized operator associated with $g$ is
$$L_{g}=-\partial_{tt}-1+3g^{2}=-\partial_{tt}+2g^{2}-(1-g^{2}).$$
Since
$$-\partial_{tt}+2g^{2}=-(\partial_{t}-\sqrt{2}g)(\partial_{t}+\sqrt{2}g)+\sqrt{2}g',$$
we can rewrite $L_{g}$ as
$$L_{g}=-(\partial_{t}-\sqrt{2}g)(\partial_{t}+\sqrt{2}g)-(1-g^{2})+\sqrt{2}g'.$$

Our main objective in this section is to derive a three-dimensional counterpart of the above decomposition. One of the main tools we used involves the Pauli matrices:
\begin{equation*}
 \sigma_1=\begin{pmatrix}0&1\\1&0\end{pmatrix},\quad
 \sigma_2=\begin{pmatrix}0&-i\\i&0\end{pmatrix},\quad
 \sigma_3=\begin{pmatrix}1&0\\0&-1\end{pmatrix}.
\end{equation*}
For a real vector $v=(v_1,v_2,v_3)$, we set
\begin{equation*}
 \sigma\cdot v=v_1\sigma_1+v_2\sigma_2+v_3\sigma_3.
\end{equation*}
Thus $\sigma\cdot v$ is a $2\times2$ matrix. For $i=1, 2, 3$, we know from \cite[Section 1]{BE01} that the Pauli matrix $\sigma_{i}$ is Hermitian
and satisfy
$$\sigma_j\sigma_k+\sigma_k\sigma_j=2\delta_{jk}I_2,$$
where $I_2$ is the $2\times2$ identity matrix.
These relations imply
$$(\sigma\cdot v)^2=|v|^2I_2.$$
Thus $\sigma\cdot v$ has eigenvalues $\pm|v|$.

Similarly, we introduce the notation
\begin{equation*}
 \sigma\cdot\nabla
 =\sigma_1\partial_{x_1}+\sigma_2\partial_{x_2}+\sigma_3\partial_{x_3}.
\end{equation*}
Acting on smooth compactly supported $\mathbb C^2$-valued functions, we have
\begin{equation}\label{first}
 \begin{split}
 (\sigma\cdot\nabla)^2&=\Delta I_2,\\
 (\sigma\cdot\nabla\mp\sqrt2u)^*
   &=-(\sigma\cdot\nabla\pm\sqrt2u).
 \end{split}
\end{equation}
Here the adjoint is taken in unweighted $L^2$. The minus sign in the adjoint formula comes from
integration by parts.

 Direct computation tells us that for a smooth
compactly supported function $\Xi\in C_c^\infty(R^3;\mathbb C^2)$, we have
\begin{equation*}
 (\sigma\cdot\nabla)(u\Xi)
   =(\sigma\cdot\nabla u)\Xi+u(\sigma\cdot\nabla)\Xi.
\end{equation*}
It follows  that
\begin{equation*}
 (\sigma\cdot\nabla\mp\sqrt2u)^*
       (\sigma\cdot\nabla\mp\sqrt2u)=-\Delta I_2+2u^2I_2
                 \pm\sqrt2\,\sigma\cdot\nabla u.
\end{equation*}From this we get
\begin{equation}\label{factor}
 L_uI_2=(\sigma\cdot\nabla\mp\sqrt2u)^*
       (\sigma\cdot\nabla\mp\sqrt2u)
 -(1-u^2)I_2\mp\sqrt2\,\sigma\cdot\nabla u.
\end{equation}
Since $u$ satisfies the Allen--Cahn equation, we know that
\begin{equation*}
 \begin{split}
 (\sigma\cdot\nabla)(\sigma\cdot\nabla u)&=(u^3-u)I_2,\\
 (\sigma\cdot\nabla)\left(\frac{1-u^2}{\sqrt2}I_2\right)
   &=-\sqrt2u\,\sigma\cdot\nabla u.
 \end{split}
\end{equation*}
Substituting these two identities into the first-order equation
shows that
\begin{equation}\label{columns}
 (\sigma\cdot\nabla\mp\sqrt2u)
 \left(\frac{1-u^2}{\sqrt2}I_2\mp\sigma\cdot\nabla u\right)=0.
\end{equation}

Let $\Psi$ be any smooth solution of
$$(\sigma\cdot\nabla\mp\sqrt2u)\Psi=0,$$
and let $\chi$ be a real,
smooth, compactly supported cutoff. The product rule yields
$$(\sigma\cdot\nabla\mp\sqrt2u)(\chi\Psi)
=(\sigma\cdot\nabla\chi)\Psi.$$
Note that
$$|(\sigma\cdot\nabla\chi)\Psi|^2=|\nabla\chi|^2|\Psi|^2.$$

For a complex scalar function, $Q_u$ denotes the sum of its values
on the real and imaginary parts. Applying \eqref{factor} to
$\chi\Psi$ and adding the quadratic forms of its two components
therefore yields
\begin{equation}\label{dircut}
 \begin{split}
 \sum_{\ell=1}^2Q_u((\chi\Psi)_\ell)
 &=\int_{\mathbb{R}^3}|\nabla\chi|^2|\Psi|^2\,dx\\
 &\quad-\int_{\mathbb{R}^3}\chi^2
   \left\langle
    \bigl((1-u^2)I_2\pm\sqrt2\,\sigma\cdot\nabla u\bigr)\Psi,
    \Psi\right\rangle\,dx.
 \end{split}
\end{equation}
Since $u$ has positive index,
Modica's estimate \cite{Modica85} and its equality case
\cite{CGS94} imply that
$$\frac{1}{2}|\nabla u|^{2}\leq \frac{1}{4}(1-u^{2})^{2},\qquad\text{in}~\mathbb{R}^{n}.$$
It follows that the matrix in the
second integral is positive definite.
The first integral in the right hand side of \eqref{dircut} is the cutoff error.
On a fixed compact set where the function $\Psi$
is nontrivial, this positivity keeps the second integral bounded away
from zero once the cutoffs tend locally to one.
If the cutoff error tends to zero, the resulting test functions have
strictly negative total second variation. This negative term
distinguishes \eqref{dircut} from the identity for solutions of the linearized equation
\eqref{jacut}, which contains only the cutoff error.

The next theorem relates the number of linearly independent
solutions of the first-order equation to the Morse index of $u$.
 After suitable cutoffs, finitely many linearly
independent solutions therefore produce a subspace on which
the total second variation is negative.
Since $\Psi$ has two complex-valued components, the sum of
the two scalar quadratic forms has complex index $2I$.
\begin{theorem}\label{dim}
Assume that $u$ satisfies the hypotheses of Theorem~\ref{main}.
Let $\Psi_1,\ldots,\Psi_d:\mathbb{R}^3\to\mathbb C^2$
be smooth solutions of
\begin{equation}\label{one oder equation}
 (\sigma\cdot\nabla\mp\sqrt2u)\Psi_k=0,
 \qquad k=1,\ldots,d,
\end{equation}
with the same choice of sign.
Suppose that these solutions are linearly independent over
$\mathbb C$.

For each $k$, assume that there are real cutoffs
$\chi_{k,j}\in C_c^\infty(\mathbb{R}^3)$ such that
$0\le\chi_{k,j}\le1$, $\chi_{k,j}\to1$ locally as
$j\to\infty$, and
\begin{equation}\label{cutoff}
 \int_{\mathbb{R}^3}
 |\nabla\chi_{k,j}|^2|\Psi_k|^2\,dx
 \longrightarrow0.
\end{equation}
Then
\begin{equation*}
 d\le2I.
\end{equation*}

If $I=1$, every smooth solution $\Psi$ of the same
first-order equation satisfying this cutoff condition
can be expressed as
\begin{equation*}
 \Psi=
 \left(
 \frac{1-u^2}{\sqrt2}I_2
 \mp\sigma\cdot\nabla u
 \right)c
\end{equation*}
for some constant vector $c\in\mathbb C^2$.
\end{theorem}

\begin{proof}
The strategy is to localize any finite independent family, obtain a
uniform negative bound on its span, and compare with the index of
the two-component quadratic form.

Suppose $\Psi_1,\ldots,\Psi_d$ are independent solutions of
$(\sigma\cdot\nabla-\sqrt2u)\Psi=0$ satisfying \eqref{cutoff}.  For $c=(c_1,\ldots,c_d)$ with $\sum|c_k|^2=1$, set
$$T_jc=\sum_kc_k\chi_{k,j}\Psi_k,$$
where cutoffs $\{\chi_{k,j}\}$ may be
chosen separately.
Then we have
\begin{equation}\label{the cutoff error}
 \|(\sigma\cdot\nabla-\sqrt2u)T_jc\|_2^2
 \le d\sum_k\int|\nabla\chi_{k,j}|^2|\Psi_k|^2=o(1).
 \end{equation}

\textbf{Claim} 1: the restrictions of the functions $\Psi_1,\ldots,\Psi_d$ to some compact ball $K$ are linear
independently over $\mathbb{C}$.

Indeed, if they were dependent on every ball, then there exists $c_{1, j},\cdots c_{d, j}$ and $\{R_{j}\}\to \infty$ such that $\sum|c_{k, j}|^{2}=1$ and
$$\sum_{k=1}^{d}c_{k, j}\Psi_{k}=0,\qquad\text{on}~~B_{R_{j}}.$$
A convergent subsequence of unit coefficient vectors would produce a nontrivial relation on all of $\mathbb R^3$. This contradicts the assumption that $\Psi_1,\ldots,\Psi_d:\mathbb{R}^3\to\mathbb C^2$ are linearly independent over $\mathbb C$.

Since the matrix
$$\frac{1-u^2}{\sqrt2}I_2+\sigma\cdot\nabla u$$
is positive definite, then the
quadratic form$$
 \int_K\left\langle
 \left(\frac{1-u^2}{\sqrt2}I_2+\sigma\cdot\nabla u\right)
 \sum_k c_k\Psi_k,\sum_k c_k\Psi_k\right\rangle
$$
has a positive minimum on that sphere. \eqref{factor} and the vanishing first-order error imply
\begin{equation}
 \begin{split}
 \sum_{\ell=1}^2Q_u((T_jc)_\ell)
 &=\|(\sigma\cdot\nabla-\sqrt2u)T_jc\|_2^2\\
 &\quad-\int\left\langle
     \bigl((1-u^2)I_2+\sqrt2\,\sigma\cdot\nabla u\bigr)T_jc,
     T_jc\right\rangle.
 \end{split}
\end{equation}
Recall that $\chi_{k,j}\to1$ locally as
$j\to\infty$, then
$$-\int\left\langle
     \bigl((1-u^2)I_2+\sqrt2\,\sigma\cdot\nabla u\bigr)T_jc,
     T_jc\right\rangle\leq - \frac{1}{2}\int_{K}\left\langle
     \bigl((1-u^2)I_2+\sqrt2\,\sigma\cdot\nabla u\bigr)T_jc,
     T_jc\right\rangle$$
     as $j\to\infty$. Combining this with \eqref{the cutoff error}, we deduce that
     $$\sum_{\ell=1}^2Q_u((T_jc)_\ell)<0$$
    for $j$ large enough. The real scalar quadratic form has index $I$, its
complexification has complex index $I$, and the direct sum of
two scalar forms has complex index $2I$.
The localized span is a negative complex subspace, so $d\le2I$.

\textbf{Claim} 2: the matrix in \eqref{columns} yields a solution of the equation \eqref{one oder equation} satisfying \eqref{cutoff}.

The columns in \eqref{columns} with the minus sign solve
the first-order equation. Because the matrix is pointwise invertible, they are linear independent over $\mathbb{C}$. The squared Hilbert--Schmidt norm of either matrix in  \eqref{columns} is exactly
\begin{equation*}
 2|\nabla u|^2+(1-u^2)^2
 =4\left(\frac12|\nabla u|^2+W(u)\right).
\end{equation*}
In particular, each column has quadratic $L^2$ growth. Let $\chi_R$ be a logarithmic cutoff between $B_R$ and $B_{R^2}$, with
$|\nabla\chi_R|\le C/(|x|\log R)$.  Integrating in the radius
therefore yields $$
 \int|\nabla\chi_R|^2
 \left|\frac{1-u^2}{\sqrt2}I_2-\sigma\cdot\nabla u\right|^2
 \le\frac{C}{\log R},$$
where the matrix norm is Hilbert--Schmidt.
Thus both columns satisfy \eqref{cutoff}.

When $I=1$, any additional independent function in this kernel would
contradict $d\le2$, proving the second statement in Theorem \ref{dim}.

For the other sign, we can use the plus-sign columns in
\eqref{columns} and the corresponding identity in
\eqref{factor}. This finishes the proof of the theorem.
\end{proof}
\section{Upper bound of the ends}\label{ends}

In this section, our main objective is to prove $N\le2I$ by comparing the number
of ends with the dimension of a space of adjoint solutions of
\begin{equation*}
 (\sigma\cdot\nabla\mp\sqrt2u)\Psi=0.
\end{equation*}

The proof will be divided into two steps. First, we bound this dimension by $2I$. According to
Theorem~\ref{dim}, this bound can be verified once we can check its
cutoff condition. To do this, we study the behavior of adjoint
solutions near the ends. We first improve the approximation
of $u$ by one-dimensional transitions and use this approximation
to derive weighted estimates for the first-order operator.
These estimates show that the adjoint solutions satisfying certain
 integrability condition  have
constant leading coefficients at the ends and decaying remainders.
The resulting expansion allows us to verify the cutoff condition.
Theorem~\ref{dim} then bounds the complex dimension of this
adjoint solution space by $2I$.

Our second task is to show that the same adjoint solution
space has dimension at least $N$.
We begin by constructing one approximate first-order solution
on each end. These functions have distinct leading terms of
size $R^{-1}$, so they provide $N$ independent coefficients.
We seek corrections that turn their linear combinations
into global solutions. The weighted estimates establish
closed range, which characterizes the existence of these
corrections by orthogonality to the adjoint solutions.
Integration by parts expresses these orthogonality conditions
in terms of the constant end coefficients obtained above.
The corrected solutions can then be transformed into
independent adjoint solutions whose constant end coefficients
all vanish. Thus the correction conditions restrict the
initial $N$ coefficients, while the vanishing conditions
restrict the adjoint solution space. Both restrictions are
determined by the same constant end coefficients and remove
the same number of dimensions. Comparing the remaining
dimensions therefore shows that the adjoint solution space
has dimension at least $N$. Combining this lower bound with
the upper bound from the first step yields $N\le2I$.

\begin{proposition}\label{corprop}
For the solutions in Proposition~\ref{geom}, we have$$
 \|u-U_f\|_{L^2(\mathbb R_z)}
       +\|u-U_f\|_{L^\infty(\mathbb R_z)}\le C r^{-2}.$$
\end{proposition}

\begin{proof}
Recall that $\psi=u-U_f$ is orthogonal to every $\gamma_i$ and
$\|\psi\|_\infty=o(1)$, as established in Proposition~\ref{geom}
using \cite[Lemma 4.11 and equation (62)]{FS26}.
The one-dimensional coercivity estimate in
\cite[Theorem A.3]{WW19b}, rescaled to our normalization,
extends to the separated layers by localization near each
transition and positivity of the potential in the intervening
well regions. Thus, for all sufficiently large $r$,
\begin{equation}\label{ncoerc}
 \left\langle(-\partial_z^2+W''(U_f))v,v\right\rangle
       \ge c_0\|v\|_2^2,\qquad
\text{for}~v\in H^1(\mathbb R_z),\quad
 \langle v,\gamma_i\rangle=0\quad(1\le i\le N),
\end{equation}
where $c_0>0$ is independent of $r$.

Differentiating \eqref{Uf} and using the exponential decay  of $g$
and \eqref{logends}--\eqref{pi}, we obtain
\begin{equation*}
 \begin{split}
 \|\Delta U_f-W'(U_f)\|_2
 &\le C\sum_{i=1}^N\bigl(|\Delta f_i|+|\nabla f_i|^2\bigr)\\
 &\quad+C\sum_{i=1}^{N-1}(1+D_i)^{1/2}e^{-\sqrt2D_i}
 \le Cr^{-2}.
 \end{split}
\end{equation*}
The phase-decay estimates in Proposition~\ref{geom} imply
$\|\psi\|_2=O(\sqrt{\log r})$. Pairing the equation for $\psi$
with $\psi$, applying \eqref{ncoerc}, and absorbing the quadratic
remainder yields
\begin{equation*}
 (-\Delta_y+c_0/2)\|\psi\|_2\le Cr^{-2}
\end{equation*}
in the distributional sense, with the norm regularized at zero.
Comparison with $Cr^{-2}+\delta r$, followed by $\delta\downarrow0$,
proves the $L^2$ bound. Interior $H^2$ estimates for the scalar
equation on unit balls in $\mathbb R^3$, followed by Sobolev
embedding, then prove the $L^\infty$ bound. This completes the proof.
\end{proof}

We next describe the first-order operator in coordinates adapted
to an end. For this purpose, we  set $$t=z-f_i(y),\qquad
\nu_i=(1+|\nabla f_i|^2)^{1/2},$$ and
$n_i=(-\nabla f_i,1)/\nu_i$. We denote
\begin{equation}\label{spin}
 S_i=\frac{I_2+\sigma_3(\sigma\cdot n_i)}
              {\sqrt{2(1+n_{i,3})}},
\end{equation}
then it is an unitary matrix. Moreover, it
satisfies
$$
 S_i\left(\sigma_3-\sigma_1f_{i,1}-\sigma_2f_{i,2}\right)S_i^{-1}
       =\nu_i\sigma_3,\qquad
 S_i-I_2=O(r^{-1}),\quad \nabla S_i=O(r^{-2}).$$
Conjugating
$\sigma\cdot\nabla\mp\sqrt2u$ by $S_i$ produces$$
\begin{aligned}
\mathcal P_i=&\mathcal P_i^0+
      \sum_{\alpha=1}^2 E_{i,\alpha}\partial_{y_\alpha}
      +E_{i,t}\partial_t+E_{i,0},\\
 \mathcal P_i^0=&\sigma_1\partial_{y_1}
       +\sigma_2\partial_{y_2}+\sigma_3\partial_t\mp\sqrt2q_i.
    \end{aligned}$$
Here
\begin{equation}\label{Ecoeff}
 \begin{split}
 E_{i,\alpha}&=S_i\sigma_\alpha S_i^{-1}-\sigma_\alpha=O(r^{-1}),\\
 E_{i,t}&=(\nu_i-1)\sigma_3=O(r^{-2}),\\
 E_{i,0}&=\sum_{\alpha=1}^2
       S_i\sigma_\alpha\partial_{y_\alpha}S_i^{-1}
                \mp\sqrt2(u-q_i)I_2.
 \end{split}
\end{equation}
The vertical interval around the $i$th end extends to fixed-width bands around the
midpoints between neighboring heights. Outside that interval, the
coefficient $u$ is extended by a cutoff to agree with $q_i$.
Observe that on this auxiliary normal line, $\|E_{i,0}\|_\infty=o(1)$.

We compute the $L^2$ kernel of the one-dimensional operator before using it in the end
decomposition. For an increasing layer and the minus sign, the
normal operator and its unweighted adjoint are
\begin{equation*}
 \begin{split}
 \sigma_3\partial_t-\sqrt2g
   &=\operatorname{diag}(\partial_t-\sqrt2g,-\partial_t-\sqrt2g),\\
 (\sigma_3\partial_t-\sqrt2g)^*
   &=\operatorname{diag}(-\partial_t-\sqrt2g,\partial_t-\sqrt2g).
 \end{split}
\end{equation*}
Since $g''=-\sqrt2gg'$, the square-integrable solution of
$(\partial_t+\sqrt2g)v=0$ is a multiple of $g'$.
The opposite scalar equation has solutions proportional to
$\cosh^2(t/\sqrt2)$ and has trivial $L^2$ kernel.
Thus the $L^2$ kernel of the one-dimensional operator is generated by $g'e_2$, and its adjoint
kernel by $g'e_1$.
We normalize these generators as $\phi_i$ and $\phi_i^*$.
Changing the layer orientation or the operator sign exchanges
the two coordinate vectors. The four choices are
\begin{equation}\label{kernels}
\begin{array}{c|c|c|c|c}
 \text{Operator}&q_i&\phi_i&\phi_i^*&\partial_*\\ \hline
 (\sigma\cdot\nabla-\sqrt2u)&g&g' e_2/\sqrt{\sigma_0}&
        g' e_1/\sqrt{\sigma_0}&\partial_w\\
 (\sigma\cdot\nabla-\sqrt2u)&-g&g' e_1/\sqrt{\sigma_0}&
        g' e_2/\sqrt{\sigma_0}&\partial_{\bar w}\\
 (\sigma\cdot\nabla+\sqrt2u)&g&g' e_1/\sqrt{\sigma_0}&
        g' e_2/\sqrt{\sigma_0}&\partial_{\bar w}\\
 (\sigma\cdot\nabla+\sqrt2u)&-g&g' e_2/\sqrt{\sigma_0}&
        g' e_1/\sqrt{\sigma_0}&\partial_w
\end{array}
\end{equation}
Proposition~\ref{corprop}  implies
\begin{equation}\label{Egen}
 \|E_{i,0}\phi_i\|_2+
       \|E_{i,0}^*\phi_i^*\|_2\le Cr^{-2}.
\end{equation}

 On an increasing
layer, the normal part of $\sigma\cdot\nabla-\sqrt2u$ is
$\sigma_3\partial_t-\sqrt2g(t)$. With $s=t/\sqrt2$, we have
$$
  (\sigma_3\partial_t-\sqrt2g)^*
                      (\sigma_3\partial_t-\sqrt2g)
  =\diag\left(-\tfrac12\partial_s^2+2-\operatorname{sech}^2s,\,
           -\tfrac12\partial_s^2+2-3\operatorname{sech}^2s\right).
  $$
Indeed, the lower bound on the entire orthogonal
complement follows from the scalar factorizations
\begin{equation*}
 \begin{split}
 -\tfrac12\partial_s^2+2-\operatorname{sech}^2s
   &=\tfrac12(\partial_s+\tanh s)^*
                    (\partial_s+\tanh s)+\tfrac32,\\
 -\tfrac12\partial_s^2+2-3\operatorname{sech}^2s
   &=\tfrac12(\partial_s+2\tanh s)^*
                    (\partial_s+2\tanh s),\\
 \tfrac12(\partial_s+2\tanh s)(\partial_s+2\tanh s)^*
   &=-\tfrac12\partial_s^2+2-\operatorname{sech}^2s.
 \end{split}
\end{equation*}
The nonzero spectra of an operator times its adjoint and the
reversed product agree. Hence the second component is bounded
below by $3/2$ off its kernel, generated by
$\operatorname{sech}^2s$. The first component already has this
lower bound. Its eigenfunction $\operatorname{sech}s$ shows
that the gap is exactly $3/2$, and the limiting potential places
the essential spectrum at $2$. Reversing the layer or the two
components treats the other cases.
The horizontal operator also has an explicit action. In the
increasing, minus-sign case,
\begin{equation*}
 (\sigma_1\partial_{y_1}+\sigma_2\partial_{y_2})(a\phi_i)
       =2(\partial_w a)\phi_i^*.
\end{equation*}
For the function orthogonal to $\phi_i$, the normal part is orthogonal to
$\phi_i^*$ by its adjoint kernel equation. The horizontal part
has the same orthogonality, because
$\sigma_1\phi_i^*$ and $\sigma_2\phi_i^*$ are constant multiples
of $\phi_i$. Thus, for $v=a\phi_i+q$ with $q\perp\phi_i$ on
each normal line,
\begin{equation}\label{flatdec}
 \mathcal P_i^0(a\phi_i)=2(\partial_*a)\phi_i^*,\qquad
 \langle\mathcal P_i^0q,\phi_i^*\rangle=0.
\end{equation}
For compact horizontal support,
$\|2\partial_*a\|_2^2=\|\nabla_y a\|_2^2$, while
\begin{equation}\label{flatgap}
 \|\mathcal P_i^0q\|_2^2
  =\|\nabla_yq\|_2^2+\|\mathcal P_{i,N}^0q\|_2^2
  \ge c\bigl(\|q\|_2^2+\|\partial_tq\|_2^2+\|\nabla_yq\|_2^2\bigr).
\end{equation}
The mixed terms cancel by the Pauli relations. Thus the estimate
controls the horizontal derivatives as well as the function and its
normal derivative.

We choose a square partition of unity adapted to the regions around the ends.
For a fixed large $L$, a band around
$(f_i+f_{i+1})/2$ has width $2L$.
On it the two adjacent cutoffs are
$$\cos\vartheta((z-(f_i+f_{i+1})/2)/L),\qquad
\sin\vartheta((z-(f_i+f_{i+1})/2)/L),$$where $\vartheta$ increases smoothly
from $0$ to $\pi/2$ on $[-1,1]$.
The outer bands are centered at
$f_1-L_1\log r$ and $f_N+L_1\log r$.
They separate the regions around the lowest and highest ends from the two well regions. Let us take$$
 \beta=\min_i(p_i-2)>0,\qquad \sqrt2L_1>3+\beta.$$
After interpolation on a compact cylinder, the partition satisfies$$\zeta_K^2+\zeta_{\mathrm{ph}}^2+\sum_i\zeta_i^2=1.$$
The notation $\zeta_{\mathrm{ph}}$ represents the two outer
pieces, whose squared norms are added. These pieces also cover
vertical escape over bounded horizontal sets.
On each transition band, normal cutoff derivatives are $O(L^{-1})$
and horizontal derivatives are $O((Lr)^{-1})$.  For $1\le i\le N-1$, we set
$I_i=e^{-\sqrt2D_i}$, where $D_i=f_{i+1}-f_i$.
On the band centered at $(f_i+f_{i+1})/2$, we have
\begin{equation}\label{collar}
 \begin{aligned}
 &|\phi_i|+|\phi_{i+1}|\le C_L I_i^{1/2},
 \qquad
 |\phi_i|\,|\phi_{i+1}|\le C_L I_i,\\
 &I_i\le Cr^{-2-\beta}.
 \end{aligned}
\end{equation}
Here $C_L$ may depend on the fixed band width $2L$.

For a function $\Psi$, let us decompose on the auxiliary normal line$$
 S_i\zeta_i\Psi=a_i\phi_i+\eta_i,\qquad
 a_i=\langle S_i\zeta_i\Psi,\phi_i\rangle,\qquad
 \eta_i\perp\phi_i.$$
The localized function is extended by zero outside the region around that end.
Thus $\eta_i$ includes the tail $-a_i\phi_i$ there.

\begin{proposition}\label{unwprop}There holds
\begin{equation}\label{unw}
 \begin{split}
 &\sum_i\left(\|\nabla_y a_i\|_2^2+\|\eta_i\|_{H^1}^2\right)
      +\|\zeta_{\mathrm{ph}}\Psi\|_{H^1}^2
 \\
 &\qquad\le C\left(\|(\sigma\cdot\nabla\mp\sqrt2u)\Psi\|_2^2+
      \|\langle r\rangle^{-1}\Psi\|_2^2+\|\Psi\|_{L^2(K)}^2\right).
 \end{split}
\end{equation}
\end{proposition}

\begin{proof}
We first prove the estimate on the individual regions around the ends and the two well regions. We
then sum the local estimates through the square partition and absorb
the overlap terms.

Now we would like to apply \eqref{flatdec}--\eqref{flatgap} to each such region.
The errors acting on $a_i\phi_i$ are bounded by
$C(r^{-1}|\nabla a_i|+r^{-2}|a_i|)$, by
\eqref{Ecoeff} and \eqref{Egen}.
The errors on $\eta_i$ are $o_{R_0}(1)\|\eta_i\|_{H^1}$ and can
be absorbed. On the well regions, $u$ approaches $1$ or $-1$,
so the corresponding coercive estimate holds. Observe that $$
  \sum_j\|(\sigma\cdot\nabla\mp\sqrt2u)(\zeta_j\Psi)\|_2^2
  =\|(\sigma\cdot\nabla\mp\sqrt2u)\Psi\|_2^2+
       \int\sum_j|\nabla\zeta_j|^2|\Psi|^2. $$
The mixed terms vanish because $\sum_j\zeta_j\nabla\zeta_j=0$.
On transition bands, reconstruction from the finitely overlapping
pieces bounds the complementary error by
$CL^{-2}\sum_j\|\eta_j\|_2^2$, which is absorbed by taking $L$
large. The normal part is bounded by
$C\sum_j\int r^{-2-\beta}|a_j|^2$, using
\eqref{collar}. This is controlled by the weighted term on
the right of \eqref{unw}. Interior elliptic estimates
handle the compact cylinder. Every step concerns compactly
supported functions, so the estimate assumes only local regularity.
\end{proof}

The estimate controls horizontal derivatives of $a_i$, whereas the
spectral gap also controls the size of $\eta_i$. We therefore use
different weights for these two terms. Put $\ell=\log r$. We denote
$$r^\lambda H_b^k:=\{a:r^{-\lambda}a\in H^k(d\ell\,d\theta)\}.$$
For functions on the normal lines, $r^\lambda L_b^2$ is defined using
$d\ell\,d\theta\,dt$. Now let us choose
\begin{equation}\label{weights}
 0<\delta<\min\{1,\beta/2\},\qquad \mu=-1-\delta.
\end{equation}
Let $\mathcal D_\mu$ be the completion of compactly supported
smooth functions in the norm
\begin{equation}\label{Dnorm}
 \begin{split}
 \|\Psi\|_{\mathcal D_\mu}^2
 &=\|\zeta_K\Psi\|_{H^1}^2+
    \sum_i\|a_i\|_{r^\mu H_b^1}^2\\
 &\quad+\sum_i\int r^{-2\mu}
     \left(|\eta_i|^2+|\partial_t\eta_i|^2
                       +|\nabla_y\eta_i|^2\right)\,dy\,dt\\
 &\quad+\int\langle r\rangle^{-2\mu}
     \left(|\zeta_{\mathrm{ph}}\Psi|^2
                 +|\nabla(\zeta_{\mathrm{ph}}\Psi)|^2\right)\,dx.
 \end{split}
\end{equation}
We set$$
 \mathcal Y_\mu=
 \{F:\langle r\rangle^{-\mu}F\in L^2(\mathbb R^3)\}.$$
The coefficients $a_i$ have cylindrical weight $\mu$.
Since $dy=r^2d\ell\,d\theta$, the functions orthogonal to $\phi_i$ in \eqref{Dnorm} have cylindrical
weight $\mu-1$.

These completions are spaces of distributions on $\mathbb R^3$.
Indeed, the reconstruction identity is
\begin{equation}\label{recon}
 \Psi=\zeta_K^2\Psi+\zeta_{\mathrm{ph}}^2\Psi+
       \sum_i\zeta_iS_i^{-1}(a_i\phi_i+\eta_i).
\end{equation}
It controls local distributional convergence. Conversely, the
coefficients are determined by a locally defined function on bounded horizontal
sets, because the end cutoffs there have bounded vertical support.
The coefficient maps and the compact terms are continuous.
A sequence with distributional limit zero therefore has zero
limit in each component whenever it converges in the completed
norm. Thus distinct elements of the completed space determine distinct
distributions.

To proceed, let us recall that weighted estimates for elliptic operators on noncompact manifolds which
were developed by Lockhart and McOwen \cite{LM85}.
For the operator considered here, we establish the required estimate
directly by applying Fourier analysis to the coefficients $a_i$
and the spectral gap estimate to the functions orthogonal to $\phi_i$.

\begin{proposition}\label{clrange}
The map $(\sigma\cdot\nabla-\sqrt2u):\mathcal D_\mu\to\mathcal Y_\mu$ is bounded and satisfies
\begin{equation}\label{clest}
 \|\Psi\|_{\mathcal D_\mu}
 \le C\left(\|(\sigma\cdot\nabla-\sqrt2u)\Psi\|_{\mathcal Y_\mu}
                      +\|\Psi\|_{L^2(K)}\right).
\end{equation}
Moreover, its kernel is finite-dimensional and its range is closed.
\end{proposition}

\begin{proof}
We prove the weighted estimate for the  coefficient $a$ in the flat case,
then absorb the geometric and localization errors. If $a=r^\mu v$,
then$$
 2r^{1-\mu}e^{i\theta}\partial_wa
       =(\partial_\ell+\mu-i\partial_\theta)v.$$
The Fourier symbol is $i\xi+\mu+k$, $k\in\mathbb Z$.
For this fixed noninteger weight,
\begin{equation*}
 |i\xi+\mu+k|^2
       =\xi^2+(\mu+k)^2
       \ge c_\mu(1+\xi^2+k^2).
\end{equation*}
Fourier inversion therefore bounds the full-cylinder $H^1$
norm by the $L^2$ norm of the image.
Cutting off at the inner boundary proves the exterior estimate
with an additional $H^1$ norm on a fixed inner annulus.
The formula for $\partial_{\bar w}$ has the same consequence.

For the complement, apply \eqref{flatgap} to $r^{-\mu}q$.
The commutator with the weight is $O(|\mu|r^{-1}q)$, which is
absorbed for large $R_0$. The flat images of $a\phi_i$ and $q$
are orthogonal by \eqref{flatdec}, so these two estimates
can be combined without cancellation.

The errors contributed by $a_i\phi_i$ are
$r^{-2}a_i+r^{-1}\nabla a_i$. In the
$\mathcal Y_\mu$ norm, they are smaller than the normal domain
norm by a factor $O(R_0^{-1})$.
The complementary errors tend to zero in operator norm by
\eqref{Ecoeff}. A normal cutoff error has size
$I_j^{1/2}a_i$, whose relative size is
$rI_j^{1/2}=O(R_0^{-\beta/2})$.
The complementary cutoff error is $O(L^{-1})$.
The reconstruction \eqref{recon} has at most two
overlapping end pieces outside the compact set.
Its derivatives are controlled by the same estimates, without
division by a cutoff. First fix $L$ sufficiently large and then
$R_0$ sufficiently large. All exterior errors are absorbed.
Local elliptic estimates on the remaining compact set prove
\eqref{clest}. The same reconstruction proves
boundedness of $(\sigma\cdot\nabla-\sqrt2u)$ between the stated spaces.

The term on $K$ is compact by the local $H^1$ bound and
Rellich's theorem. The compact-remainder estimate \eqref{clest} therefore implies a
finite-dimensional kernel and closed range.
\end{proof}

  The dual of
$\mathcal Y_\mu$ under the unweighted integral pairing consists
of functions $\Phi$ with $\langle r\rangle^\mu\Phi\in L^2$.
An annihilator of the range satisfies the distributional formal
adjoint equation from \eqref{first}. The dual weight
is therefore imposed on the function, without changing that
unweighted differential expression. Since $\mu=-1-\delta$,
the two conditions are
\begin{equation}\label{adjdual}
 (\sigma\cdot\nabla+\sqrt2u)\Phi=0,\qquad
 \langle r\rangle^{-1-\delta}\Phi\in L^2(\mathbb R^3).
\end{equation}
Elliptic regularity makes it smooth.
Decompose the localized functions obtained from the solution of $(\sigma\cdot\nabla+\sqrt2u)\Phi=0$ as
$$S_i\zeta_i\Phi=b_i\phi_i+\rho_i.$$

Proposition~\ref{clrange} establishes finite-dimensional kernel
and closed range. To complete the Fredholm argument, we must
also prove that the cokernel is finite-dimensional.
The preceding description of the cokernel reduces this task
to estimating the dimension of the adjoint solutions
satisfying \eqref{adjdual}.
We seek to apply Theorem~\ref{dim}, which requires these
solutions to satisfy the cutoff condition.
The next proposition verifies this requirement by determining
their behavior near the ends.
More precisely, on each end we decompose an adjoint solution
into a constant multiple of the square-integrable solution
of the the one-dimensional equation in $t$  and a remainder with weighted decay.
The constant terms have quadratic $L^2$ growth, so logarithmic
cutoffs make their cutoff errors tend to zero.
The decay estimates yield the same conclusion for the remainders.
Theorem~\ref{dim} therefore bounds the complex dimension
of the adjoint solutions by $2I$.

\begin{proposition}\label{adjprop}
Every function satisfying \eqref{adjdual} has
\begin{equation}\label{adjasym}
 b_i=b_{i,0}+\widetilde b_i,\qquad
 \widetilde b_i\in r^{\delta-1}H_b^1,\qquad
 \rho_i,\ \partial_t\rho_i,\ \nabla_y\rho_i
                       \in r^{\delta-1}L_b^2,
\end{equation}
where $b_{i,0}$ is constant. The same complementary estimate
holds on the well pieces. Such a function satisfies
\eqref{cutoff}. Consequently, $(\sigma\cdot\nabla-\sqrt2u):\mathcal D_\mu\to\mathcal
Y_\mu$ is a Fredholm operator and
\begin{equation}\label{fincok}
 \dim_{\mathbb C}\coker
       ((\sigma\cdot\nabla-\sqrt2u):\mathcal D_\mu\longrightarrow\mathcal Y_\mu)\le2I.
\end{equation}
\end{proposition}

\begin{proof}
We first obtain derivative estimates directly from the weighted
integrability in \eqref{adjdual}. These estimates will allow us to
project the adjoint equation and derive the end expansion.

Let $v=\langle r\rangle^{-\delta}$ and let $\chi_R(y)$ equal
one for $r<R$ and zero for $r>2R$.
Apply Proposition~\ref{unwprop} to $v\chi_R\Phi$.
A vertical cutoff can first be imposed and then removed at fixed
$R$, since \eqref{adjdual} controls the whole vertical
line over bounded horizontal sets. Recall that $\Phi$ satisfies the equation
$$(\sigma\cdot\nabla+\sqrt2u)\Phi=0,$$
we can get $$
 \|(\sigma\cdot\nabla+\sqrt2u)(v\chi_R\Phi)\|_2^2
 =\int|\nabla(v\chi_R)|^2|\Phi|^2
 \le C\|\langle r\rangle^{-1-\delta}\Phi\|_2^2.$$
The other terms on the right of \eqref{unw} are
uniformly bounded. Multiplication by $v\chi_R$ commutes with
the $L^2(\mathbb R_t)$ projection onto $\phi_i$. Fatou's lemma, weak local compactness,
and the product rule therefore imply
\begin{equation}\label{adjfst}
 b_i\in r^\delta H_b^1,\qquad
 \rho_i,\ \partial_t\rho_i,\ \nabla_y\rho_i
                                  \in r^{\delta-1}L_b^2.
\end{equation}
The value estimate for $b_i$ follows directly from
\eqref{adjdual}, using $dy=r^2d\ell\,d\theta$.

Let us project the localized equation
$$\mathcal P_i(S_i\zeta_i\Phi)=S_i[(\sigma\cdot\nabla+\sqrt2u),\zeta_i]\Phi$$
onto $\phi_i^*$, and multiply by $re^{i\theta}$ for the $\partial_w$ term
or by $re^{-i\theta}$ for the $\partial_{\bar w}$ term. It becomes$$
 (\partial_\ell\pm i\partial_\theta)b_i=F_i,\qquad
 F_i\in r^{\delta-1}L_b^2.$$
We check the last assertion term by term.
The normal errors are bounded by
$Cr^{-1}(|b_i|+|r\partial_rb_i|+|\partial_\theta b_i|)$.
The complementary errors are bounded by$$
 C\left(\|\nabla_y\rho_i\|_2+
       r^{-1}\|\partial_t\rho_i\|_2+r^{-1}\|\rho_i\|_2\right).$$
For the zeroth-order term, this uses the adjoint estimate in
\eqref{Egen}. On a transition band, reconstruction
of the neighboring pieces and \eqref{collar} bound the
projected commutator by sums of
\begin{equation}
 CrI_j|b_k|,\qquad
 CrI_j^{1/2}\|\rho_k\|_2,\qquad
 CrI_j^{1/2}\|\zeta_{\mathrm{ph}}\Phi\|_2.
\end{equation}
Here $rI_j=O(r^{-1-\beta})$ and
$rI_j^{1/2}=O(r^{-\beta/2})$.
All terms have the stated weight by \eqref{adjfst}.
The outer bands decay faster, and the inner terms have compact
support.

We can extend $F_i$ across the finite end of the cylinder and invert
$\partial_\ell\pm i\partial_\theta$ at the noninteger weight
$\delta-1$ by Fourier transform.
This produces a particular solution in
$r^{\delta-1}H_b^1$.
The difference from $b_i$ is homogeneous and belongs to
$r^\delta L_b^2$ on the exterior. For $\partial_\ell-i\partial_\theta$, the homogeneous Fourier
terms have the form $c_k r^{-k}e^{ik\theta}$.
Membership in $r^\delta L_b^2$ and $0<\delta<1$ exclude $k<0$.
The $k=0$ term is constant, and all remaining terms decay at
least as $r^{-1}$. For the opposite sign, the allowed indices
are reversed, with the same decay conclusion.
Interior estimates on a smaller exterior cylinder show that
their sum belongs to $r^{\delta-1}H_b^1$.
This proves \eqref{adjasym}.

Let us reconstruct $\Phi$ using \eqref{recon}.
The finite sum of the terms with constant coefficients $b_{i,0}$  has bounded
$L^2$ mass on each vertical line. Its mass in $B_R$ is
therefore $O(R^2)$, and a logarithmic cutoff has error
$O(1/\log R)$.
For the remaining terms,
$(1+|x|)^{-2}|\Phi_{\mathrm{rem}}|^2$ is integrable.
Indeed, $|x|\ge r$ and the cylindrical weights in
\eqref{adjasym} have exponent $\delta-1<0$.
Over bounded horizontal sets, the well remainder is in ordinary
vertical $L^2$. Dominated convergence handles its cutoff error.
Thus \eqref{cutoff} holds for $\Phi$.
Theorem~\ref{dim} bounds the complex dimension of the adjoint kernel by $2I$.
Together with Proposition~\ref{clrange}, this proves
the Fredholm assertion.
\end{proof}

Proposition~\ref{adjprop} bounds the complex dimension of the
adjoint solution space by $2I$. We proceed to show that this dimension
is at least $N$, which will yield the desired estimate for
the number of ends.

For each end, we construct an approximate first-order solution
with a leading term of order $r^{-1}$.
We then determine which linear combinations of these functions
can be corrected to global solutions.
The Fredholm property reduces this question to orthogonality
conditions against the adjoint solutions.
Integration by parts expresses these conditions in terms of
the constant end coefficients found in Proposition~\ref{adjprop}.
For every combination satisfying these conditions, we construct
a correction and use the relation between the two first-order
equations to obtain an adjoint solution.
This construction preserves linear independence, and the
resulting adjoint solutions have vanishing constant end
coefficients.
We therefore compare the combinations that admit corrections
with the adjoint solutions whose constant coefficients vanish.
The same constant coefficients determine both sets of
restrictions, so both dimension counts lose the same number
of independent parameters.
Comparing the remaining dimensions proves that the adjoint
solution space has dimension at least $N$.

\begin{proposition}\label{endcnt}  Let $d$ be the finite complex dimension of the space of
functions satisfying \eqref{adjdual}. Then
\begin{equation}\label{adjcnt}
 N\le d.
\end{equation}
\end{proposition}

\begin{proof}
Throughout this proof, $\phi_i$ and $\phi_i^*$ denote the normalized
generators in the minus-sign rows of \eqref{kernels}.
In particular, $\phi_i^*$ also spans the square-integrable solutions of the
plus-sign of the one-dimensional equation in $t$. We use the domain $\mathcal D_\mu$ for the
minus-sign operator.

Choose an inner radial cutoff $\eta$ and define on the $i$th end
\begin{equation}\label{endterm}
 \Psi_i^0=\eta(r)\zeta_iS_i^{-1}
 \begin{cases}
       \bar w^{-1}\phi_i(t),&q_i=g,\\
       w^{-1}\phi_i(t),&q_i=-g.
      \end{cases}
\end{equation}
These choices satisfy the required Cauchy--Riemann equation
exactly. The remaining geometric error has normal-line
$L^2$ size $O(r^{-3})$, and the cutoff error is
$O(r^{-1}I_j^{1/2})$.
In the $\mathcal Y_\mu$ norm, their cylindrical sizes are
$r^{-1+\delta}$ and $r^{\delta-\beta/2}$.
Both are square integrable by \eqref{weights}.
Hence $(\sigma\cdot\nabla-\sqrt2u)\Psi_i^0\in\mathcal Y_\mu$.

These functions are independent modulo $\mathcal D_\mu$.
For the  coefficient $a_i$ of a function in that domain,
consider
\begin{equation}\label{average}
 \frac1{2\pi\log2}
 \int_R^{2R}\int_0^{2\pi}
       re^{\mp i\theta}a_i(r,\theta)\,d\theta\,\frac{dr}{r},
\end{equation}
with the minus sign for an increasing end and the plus sign
for a decreasing end. Its absolute value is bounded by
$CR^{\mu+1}\|a_i\|_{r^\mu L_b^2(R<r<2R)}$, which tends to zero.
For $\Psi_i^0$ the average tends to one, since
$\int\zeta_i^2|\phi_i|^2\,dt\to1$.
For $\Psi_j^0$ with $j\ne i$, it tends to zero by separation and
\eqref{collar}. Applying these averages to a relation modulo
$\mathcal D_\mu$ shows that every coefficient in the relation vanishes.

Let $\Phi$ satisfy \eqref{adjdual}. With the generators
fixed above, its end decomposition is
\begin{equation}\label{adjtr}
 S_i\zeta_i\Phi=b_i\phi_i^*+\rho_i,\qquad
 b_i=b_{i,0}+\widetilde b_i,\qquad \rho_i\perp\phi_i^*.
\end{equation}
The remainders satisfy \eqref{adjasym}.

We claim:
\begin{equation}\label{endpair}
 \int_{\mathbb R^3}
       \Phi^*(\sigma\cdot\nabla-\sqrt2u)\Psi_i^0\,dx
       =2\pi\,\overline{b_{i,0}}.
\end{equation}
Here $\Phi^*$ is the conjugate transpose of the column $\Phi$.

The integral converges absolutely, since its absolute integrand has
integral at most
\begin{equation*}
 \|\langle r\rangle^{-1-\delta}\Phi\|_2
 \|\langle r\rangle^{1+\delta}
               (\sigma\cdot\nabla-\sqrt2u)\Psi_i^0\|_2.
\end{equation*}

Let us choose a smooth radial cutoff $\chi_R(r)=\chi(r/R)$, where
$\chi=1$ on $[0,1]$ and $\chi=0$ on $[2,\infty)$.
The function $\chi_R\Psi_i^0$ is compactly supported in all three
coordinates. Indeed, $\eta$ excludes the inner coordinate boundary,
and each end cutoff has bounded vertical support over a bounded
horizontal annulus. The formal adjoint equation therefore implies
\begin{equation}\label{egreen}
 \int\chi_R\Phi^*(\sigma\cdot\nabla-\sqrt2u)\Psi_i^0\,dx
 =-\int\frac{\chi'(r/R)}R
                \Phi^*(\sigma\cdot e_r)\Psi_i^0\,dx,
\end{equation}
where $e_r=(\cos\theta,\sin\theta,0)$.
The radial matrix is
\begin{equation*}
 \sigma\cdot e_r=
 \begin{pmatrix}0&e^{-i\theta}\\ e^{i\theta}&0\end{pmatrix}.
\end{equation*}
The functions $\phi_i$ and $\phi_i^*$ in \eqref{kernels} consequently satisfy
\begin{equation*}
 \begin{cases}
 (\sigma\cdot e_r)(\bar w^{-1}\phi_i)=r^{-1}\phi_i^*,
       &q_i=g,\\
 (\sigma\cdot e_r)(w^{-1}\phi_i)=r^{-1}\phi_i^*,
       &q_i=-g.
 \end{cases}
\end{equation*}
Thus both layer orientations have the same sign in the pairing.

For large $R$, the cutoff $\eta$ equals one on $R<r<2R$.
Using \eqref{adjtr}, the normalization
$\|\phi_i^*\|_2=1$, and $\rho_i\perp\phi_i^*$, we obtain
\begin{equation}\label{raderr}
 \left|r\int_{\mathbb R}
       \Phi^*(\sigma\cdot e_r)\Psi_i^0\,dt-\overline{b_i}\right|
 \le Cr^{-1}\bigl(|b_i|+\|\rho_i\|_2\bigr).
\end{equation}
To estimate the error in \eqref{raderr}, express the left side
using $S_i\zeta_i\Phi$ and compare the two matrices
$S_i(\sigma\cdot e_r)S_i^{-1}$ and $\sigma\cdot e_r$.
Their difference is $O(r^{-1})$.
The remaining flat normal integral is exactly $\overline{b_i}$.

Note that the factor $r^{-1}$ in the end term cancels the radial factor
in $dx=r\,dr\,d\theta\,dt$. On this fixed-width logarithmic annulus,
weighted Cauchy--Schwarz gives
\begin{equation*}
 \frac1R\int_R^{2R}\int_0^{2\pi}|\widetilde b_i|\,d\theta\,dr
 \le CR^{\delta-1}
       \|r^{1-\delta}\widetilde b_i\|_{L_b^2(R<r<2R)}
 \longrightarrow0.
\end{equation*}
After integration against $\chi'(r/R)/R$, the error in
\eqref{raderr} is bounded by
\begin{equation*}
 CR^{-1}|b_{i,0}|+
 CR^{\delta-2}\left(
   \|r^{1-\delta}\widetilde b_i\|_{L_b^2}
   +\|r^{1-\delta}\rho_i\|_{L_b^2}\right),
\end{equation*}
which also tends to zero. Since
$-\int_1^2\chi'(s)\,ds=1$, the right side of
\eqref{egreen} tends to $2\pi\overline{b_{i,0}}$.
Absolute integrability handles the left side and proves
\eqref{endpair}.

Let $s$ be the complex dimension of the space of vectors
$(b_{1,0},\ldots,b_{N,0})$ arising from all functions satisfying \eqref{adjdual}.
The constants depend linearly on the solution and are unique by
\eqref{adjasym}. Hence the functions satisfying \eqref{adjdual} with every
constant coefficient zero form a space of dimension $d-s$.
By closed range and the annihilator description
\eqref{adjdual}, coefficients $z_1,\ldots,z_N\in\mathbb C$ admit
a correction $\psi_0\in\mathcal D_\mu$ with
\begin{equation}\label{corrfld}
 \Psi=\sum_i z_i\Psi_i^0+\psi_0,\qquad
       (\sigma\cdot\nabla-\sqrt2u)\Psi=0,
\end{equation}
exactly when$$
 \sum_i z_i\overline{b_{i,0}}=0
 \quad\text{for every }\Phi\text{ satisfying }\eqref{adjdual}.$$
Indeed, the image of the end combination must pair to zero with
every function satisfying \eqref{adjdual}. Identity \eqref{endpair} expresses this
requirement as the stated conditions. There are $s$ independent complex
conditions, so the admissible coefficient vectors form a space
of dimension $N-s$. Choose a correction for each member of a
basis and extend those choices linearly. The resulting $N-s$
exact solutions are independent by the independence modulo
$\mathcal D_\mu$ proved above.

To compare these solutions with the adjoint space, we take complex
conjugates, interchange the two components, and change the sign of
the second component. More precisely, we use the map$$
 \mathcal C
 \begin{pmatrix}\psi_1\\ \psi_2\end{pmatrix}
 =
 \begin{pmatrix}\overline{\psi_2}\\-\overline{\psi_1}\end{pmatrix}
 =i\sigma_2\overline{\Psi}.$$
This map is an anti-linear isometry. The Pauli matrices satisfy
$\mathcal C\sigma_j=-\sigma_j\mathcal C$, and $u$ is real.
Consequently,
\begin{equation}\label{conjeq}
 (\sigma\cdot\nabla+\sqrt2u)\mathcal C
       =-\mathcal C(\sigma\cdot\nabla-\sqrt2u).
\end{equation}
Applying the two sign changes to the two Pauli factors in
\eqref{spin} also shows that
$\mathcal C S_i=S_i\mathcal C$.
For the functions $\phi_i$ and $\phi_i^*$, $\mathcal C\phi_i=\phi_i^*$ for
an increasing end and $\mathcal C\phi_i=-\phi_i^*$ for a
decreasing end. It preserves the corresponding orthogonal
complements.

Every function in \eqref{corrfld} satisfies the
integrability condition in \eqref{adjdual} after applying
$\mathcal C$. For the correction this follows from
$\mathcal D_\mu\subset L^2$: the normal term has
$r^{2+2\mu}=r^{-2\delta}\le C$, and the complementary weights
in \eqref{Dnorm} are stronger than the unweighted norm.
For each end term, directly from \eqref{endterm},
\begin{equation*}
 \|\Psi_i^0(y,\cdot)\|_{L^2(\mathbb R_z)}\le Cr^{-1}
 \quad\text{for large }r.
\end{equation*}
This is square integrable with the dual weight
$\langle r\rangle^{-1-\delta}$, while its support over bounded
horizontal sets is compact.
Equation \eqref{conjeq} and elliptic regularity now
place $\mathcal C\Psi$ in the smooth adjoint solution space.

All its constant end coefficients vanish. Since $\mathcal C$
commutes with $S_i$, the coefficient of $\phi_i^*$ in the
transformed correction is the complex conjugate of the coefficient
$a_i$ in the original correction, with the sign specified above.
For the original coefficient this mean satisfies
\begin{equation*}
 \left|\frac1{2\pi\log2}
       \int_R^{2R}\int_0^{2\pi}a_i(r,\theta)\,d\theta\,\frac{dr}{r}\right|
 \le CR^\mu\|a_i\|_{r^\mu L_b^2(R<r<2R)}
 \longrightarrow0.
\end{equation*}
For every end term, the coefficient of $\phi_i$ in the localized
decomposition on the $i$th end is $O(r^{-1})$, by the vertical-line estimate and Cauchy--Schwarz.
Its ordinary annular mean also tends to zero. Finally,
\eqref{adjasym} identifies the limit of these means with the constant
coefficient of the full adjoint solution.

The map $\mathcal C$ preserves complex linear independence.
The $N-s$ corrected solutions therefore produce $N-s$
independent solutions satisfying \eqref{adjdual} in a space of dimension $d-s$.
Thus$$
 N-s\le d-s,$$
which proves \eqref{adjcnt}.
\end{proof}

The preceding dimension count can also include solutions of
$(\sigma\cdot\nabla-\sqrt2u)\Psi=0$ that belong to
$\mathcal D_\mu$.
The independence modulo $\mathcal D_\mu$ established above
shows that the span of the $N-s$ corrected solutions
intersects this kernel only at zero.
For a solution in this kernel, \eqref{conjeq} shows that its
image under $\mathcal C$ satisfies the adjoint equation.
The integrability and annular-mean estimates used above
for the correction $\psi_0$ also apply to this solution.
Its image therefore satisfies \eqref{adjdual}, with
$b_{i,0}=0$ on every end.
Since $\mathcal C$ preserves complex linear independence,
the $N-s$ corrected solutions and a basis of the weighted
kernel together yield independent solutions in this
adjoint subspace, whose dimension is $d-s$.
Consequently,
\begin{equation}\label{wker}
 N+\dim_{\mathbb C}
 \ker\bigl((\sigma\cdot\nabla-\sqrt2u):
                \mathcal D_\mu\longrightarrow\mathcal Y_\mu\bigr)
 \le d.
\end{equation}
Thus decaying first-order solutions, when present, further reduce
the number of ends permitted by the adjoint dimension bound.

We can now combine the end comparison with the adjoint dimension
bound to prove the estimate for arbitrary positive finite index.
 \begin{proof}[Proof of the end bound in Theorem~\ref{main}]
If $N\le1$, then \eqref{finends} follows from $I\ge1$.
Suppose that $N\ge2$.
Combining \eqref{wker} with \eqref{fincok}, we obtain
$$
N+\dim_{\mathbb C}
\ker\bigl((\sigma\cdot\nabla-\sqrt2u):
\mathcal D_\mu\longrightarrow\mathcal Y_\mu\bigr)
\le2I.
$$
Since the kernel dimension is nonnegative, this implies
$N\le2I$ and proves \eqref{finends}.
\end{proof}

When $I=1$, the can use the half-space theorem \cite{HLSWW21} to exclude zero or one end,
and we conclude that $N=2$.
\begin{theorem}\label{twoends}
If $I=1$, then the zero set of $u$ has exactly two ends.
\end{theorem}

\begin{proof}
The bound \eqref{finends} implies $N\le2$.
If $N=0$, the zero set is compact or empty. Similar to the proof of \cite[Lemma 3.3]{WW19}, $u$ must be a constant solution.

If $N=1$, its logarithmic height is bounded on one side of a
horizontal plane, and the compact remaining part preserves
that property. Hence the full zero set lies in an
open half-space. By the half-space theorem in
\cite{HLSWW21}, the nonconstant solution must be the
one-dimensional heteroclinic. Such a solution is stable,
contrary to index one. Therefore $N=2$. The proof is thus completed.
\end{proof}

\section{Asymptotic behavior near the ends and axial symmetry}\label{symm}
In this section, we prove the axial-symmetry of the Morse index one solution by using the moving plane method. The symmetry of two-end solutions has been studied in several works.
In \cite{GLW}, Gui, Liu, and Wei constructed and studied families of
axially symmetric two-end solutions. Subsequently, Liang and Yang
\cite[Theorem 1.4]{LY25} established axial symmetry
under a condition on the logarithmic growth rates, without assuming
symmetry in advance. More precisely, for ends with asymptotic heights
$$f_\pm(y)=c_\pm\log|y|+b_\pm+o(1),$$
their theorem is applicable when
$c_+-c_->2\sqrt2$.

For the index-one solutions considered here, the preceding section
reduces the problem to two ends, while Proposition~\ref{geom}
ensures that $c_+-c_->\sqrt2$. The range
$\sqrt2<c_+-c_-\le2\sqrt2$ is not covered by \cite[Theorem 1.4]{LY25}, therefore requires a separate argument.
We develop estimates valid throughout the range
$c_+-c_->\sqrt2$, tracing the interaction between the two ends
in the Toda equation. These estimates determine the terms
$d_i\cdot y/r^2$ in the end expansions. The balancing formula then
shows that a single horizontal translation removes these terms
from both expansions, providing the common center used in the
moving-plane argument.

After replacing $u$ by $-u$ if necessary, we may assume the outer phase is $1$.
Denote the upper and lower heights from Proposition~\ref{geom}
by $f_+$ and $f_-$.  Consider the approximation
\begin{equation}\label{twolay}
 U=1+g(z-f_+)-g(z-f_-).
\end{equation}
We set
$$\psi=u-U, \qquad p=\sqrt2(c_+-c_-)>2.$$
The functions $\gamma_+=g'(z-f_+)$ and
$\gamma_-=-g'(z-f_-)$ satisfy $\langle\psi,\gamma_\pm\rangle=0$.

We first improve the horizontal derivative estimate for the error $\psi$. This gives a Toda equation with a sufficiently
small remainder to determine the term $d_i\cdot y/r^2$ in each end's expansion.

\begin{proposition}\label{todap}
For the two ends in \eqref{twolay},
\begin{equation}\label{twocorr}
 \|\psi\|_2+\|\psi\|_\infty\le Cr^{-2},\qquad
 \|\nabla_y\psi\|_2+\|\nabla_y\psi\|_\infty\le Cr^{-3},
\end{equation}
and
\begin{equation}\label{toda}
 \Delta f_\pm=\pm12\sqrt2\,e^{-\sqrt2(f_+-f_-)}+O(r^{-4}).
\end{equation}
\end{proposition}

\begin{proof}
The first estimate in \eqref{twocorr} is Proposition~\ref{corprop}.
For the interaction term $K=U_{zz}-W'(U)$, substitution of the
explicit formula for $g$ gives
\begin{equation*}
 \begin{split}
 \langle K,\gamma_+\rangle
 &=48e^{-\sqrt2(f_+-f_-)}\bigl(1-e^{-\sqrt2(f_+-f_-)}\bigr)\\
 &\quad\times\int_0^\infty
 \frac{s^2\,ds}{(s+e^{-\sqrt2(f_+-f_-)})^2(s+1)^4}.
 \end{split}
\end{equation*}
The integral equals $1/3+O(r^{-p}\log r)$, as follows by
splitting at $s=e^{-\sqrt2(f_+-f_-)}$ and $s=1$.
Reflection in the midpoint between the two heights changes the
sign of the oriented translation. Hence
\begin{equation}\label{intproj}
 \langle K,\gamma_\pm\rangle
 =\pm16e^{-\sqrt2(f_+-f_-)}+O(r^{-2p}\log r).
\end{equation}
The same exponential tails give
\begin{equation}\label{intbnd}
 \begin{split}
 \|K\|_2+\sum_{i=\pm}\|\partial_{f_i}K\|_2
     &\le Cr^{-p}\sqrt{\log r},\\
 |\langle\gamma_+,\gamma_-\rangle|
  +|\langle\partial_z\gamma_+,\gamma_-\rangle|
     &\le Cr^{-p}\log r,\\
 \|[W''(U)-W''(g(z-f_i))]\gamma_i\|_2
     &\le Cr^{-p}\sqrt{\log r}\quad(i=\pm).
 \end{split}
\end{equation}
Here the height derivatives hold $z$ and the other height fixed.
The square-root factor comes from integration over the interval
between the ends, whose length is $O(\log r)$.

Let $P$ be the vertical orthogonal projection onto the complement
of $\gamma_+$ and $\gamma_-$. The exact equation for $\psi$ is
\begin{equation}\label{psieq}
 \begin{split}
 (-\Delta+W''(U))\psi
 &=-\sum_{i=\pm}(\Delta f_i)\gamma_i
   +\sum_{i=\pm}|\nabla f_i|^2\partial_z\gamma_i+K
   -3U\psi^2-\psi^3.
 \end{split}
\end{equation}
Differentiate at fixed $z$, with $\alpha=1,2$. Projection removes the terms
$-(\partial_\alpha\Delta f_i)\gamma_i$, which are the only
third derivatives of the heights. Since
$|\nabla f_i|=O(r^{-1})$, $|D^2f_i|=O(r^{-2})$, and
$\|\psi\|_2+\|\psi\|_\infty=O(r^{-2})$, the remaining terms
in the differentiated equation satisfy
\begin{equation*}
 \|P L_u\psi_\alpha\|_2
 \le C\bigl(r^{-3}+r^{-p-1}\sqrt{\log r}\bigr)\le Cr^{-3}.
\end{equation*}
The differentiated error has a small translation component.
Indeed, differentiating orthogonality gives
$\langle\psi_\alpha,\gamma_i\rangle
=f_{i,\alpha}\langle\psi,\partial_z\gamma_i\rangle$. We
decompose
\begin{equation*}
 \psi_\alpha=\eta_\alpha+\sum_{i=\pm}a_{\alpha,i}\gamma_i,
 \qquad P\eta_\alpha=\eta_\alpha.
\end{equation*}
The Gram matrix of the translations tends to $\sigma_0$ times
the identity. The orthogonality equations and their first
derivatives therefore imply
\begin{equation*}
 \sum_i|a_{\alpha,i}|\le Cr^{-3},\qquad
 \sum_i|\nabla a_{\alpha,i}|
 \le Cr^{-4}+Cr^{-1}\sum_{\beta=1}^2\|\eta_\beta\|_2.
\end{equation*}
In $P L_u(a_{\alpha,i}\gamma_i)$ the term
$-(\Delta a_{\alpha,i})\gamma_i$ also vanishes. The remaining
terms, using \eqref{intbnd}, yield
\begin{equation*}
 \|P L_u\eta_\alpha\|_2
 \le Cr^{-3}+Cr^{-2}\sum_{\beta=1}^2\|\eta_\beta\|_2.
\end{equation*}
The coercivity estimate \eqref{ncoerc} persists with $U$ replaced
by $u$. Apply the norm comparison from the proof of
Proposition~\ref{corprop} to the pair $(\eta_1,\eta_2)$ and
absorb the $Cr^{-2}$ term. Comparison with $Cr^{-3}$ proves
$\|\eta_\alpha\|_2\le Cr^{-3}$. The exterior comparison is
justified by the phase tails and the preliminary bound
$\|\eta_\alpha\|_2=O(\sqrt{\log r})$, with a positive multiple
of $r$ added to the comparison function and then sent to zero.
The corresponding local pointwise estimate follows from scalar
$H^2$ estimates. More explicitly, for
$P L_u w=F$, $Pw=w$, $PF=F$, recover
$L_u w=F+\sum_i\lambda_i\gamma_i$.
Twice differentiated orthogonality determines the coefficients:
\begin{equation*}
 \sum_j\langle\gamma_j,\gamma_i\rangle\lambda_j
 =2\langle\nabla_y w,\nabla_y\gamma_i\rangle
   +\langle w,\Delta_y\gamma_i\rangle
   +\langle w,(-\partial_z^2+W''(u))\gamma_i\rangle.
\end{equation*}
A local energy estimate controls $\nabla w$, and the uniformly
invertible Gram matrix controls $\lambda_i$. Interior $H^2$
estimates and Sobolev embedding then bound $\|w\|_\infty$ by
the local $L^2$ norms of $w$ and $F$.
Applied to $\eta_\alpha$, this completes \eqref{twocorr}.

Finally, let us project \eqref{psieq} onto each translation. Twice
differentiated orthogonality gives
\begin{equation*}
 \begin{split}
 \langle(-\Delta+W''(U))\psi,\gamma_i\rangle
 &=2\langle\nabla_y\psi,\nabla_y\gamma_i\rangle
   +\langle\psi,\Delta_y\gamma_i\rangle\\
 &\quad+\langle\psi,[W''(U)-W''(g(z-f_i))]\gamma_i\rangle.
 \end{split}
\end{equation*}
The first two terms and the nonlinear projection are $O(r^{-4})$.
The diagonal slope term vanishes because
$\langle\partial_z\gamma_i,\gamma_i\rangle=0$.
Using \eqref{intproj}--\eqref{intbnd} for the remaining terms,
we obtain
\begin{equation*}
 \sigma_0\Delta f_\pm
 =\pm16e^{-\sqrt2(f_+-f_-)}
   +O\bigl(r^{-4}+r^{-p-2}\log r+r^{-2p}\log r\bigr).
\end{equation*}
For fixed $p>2$, the remainder is $O(r^{-4})$.
Since $16/\sigma_0=12\sqrt2$, this proves \eqref{toda}.
\end{proof}

The leading interaction may decay too slowly to be included in
an $o(r^{-1})$ height error. We therefore subtract an exact
radial solution of the Toda system before determining the terms
$d_i\cdot y/r^2$. Now we define$$
 F_\pm(r)=c_\pm\log r+b_\pm
       \pm\frac1{\sqrt2}\log(1+\Lambda r^{-(p-2)}),
 \qquad
 \Lambda=\frac{24e^{-\sqrt2(b_+-b_-)}}{(p-2)^2}.$$
Note that
$$\Delta\log(1+\Lambda r^{-(p-2)})
=(p-2)^2\Lambda r^{-p}(1+\Lambda r^{-(p-2)})^{-2}.$$
Consequently,
$$\Delta F_\pm=\pm12\sqrt2e^{-\sqrt2(F_+-F_-)}.$$

Before proceeding, we prove a technical lemma.
\begin{lemma}\label{poisson}
Suppose $w\in C^2(\mathbb R^2\setminus\overline{B_{R_0}})$,
\begin{equation*}
 \Delta w=O(|y|^{-q}),\qquad
 \lim_{|y|\to\infty}w(y)=0,
 \qquad q>2.
\end{equation*}
Then $w=O(r^{-\beta})$ for every
$0<\beta<\min\{1,q-2\}$.
If $q>3$, there is $d\in\mathbb R^2$ such that
\begin{equation}\label{poisang}
 w=d\cdot\frac{y}{r^2}+O_{C_b^1}(r^{-1-\varepsilon})
 \quad\text{for every }0<\varepsilon<\min\{1,q-3\}.
\end{equation}
\end{lemma}

\begin{proof}
Let $v(x):=w(x/|x|^2)$ be the Kelvin transform of $w$. Then $v(x)\to0$ as $x\to0$ and
$\Delta v=O(|x|^{q-4})$, which belongs locally to $L^s$ for
every finite $s$ satisfying $s(4-q)<2$. Since $q>2$, then we can choose $s>1$. By \cite{Gi-Tr}, the local Newtonian potential of $\Delta v$ belongs to $W^{2, s}$.
Subtracting it from $v$ leaves a bounded
harmonic function, so the singularity at the origin is removable.
Interior $W^{2,s}$ regularity and Sobolev embedding imply
$v(x)=O(|x|^\beta)$ for the stated exponents.
When $q>3$, they give $v\in C^{1,\varepsilon}$ for every
$0<\varepsilon<\min\{1,q-3\}$.
Since $v(0)=0$,
$v(x)=d\cdot x+O(|x|^{1+\varepsilon})$ and
$\nabla v(x)=d+O(|x|^\varepsilon)$.
Inverting the transformation proves both conclusions,
including the scaled derivative estimate in \eqref{poisang}.
\end{proof}

Now we apply the lemma to $v_\pm=f_\pm-F_\pm=o(1)$.
Their sum satisfies $\Delta(v_++v_-)=O(r^{-4})$.
For the difference $v=v_+-v_-$, the mean value formula gives
\begin{equation}\label{poisitr}
 \Delta v+
 48e^{-\sqrt2(F_+-F_-)}
       \left(\int_0^1e^{-\sqrt2tv}\,dt\right)v=O(r^{-4}).
\end{equation}
The coefficient of $v$ is $O(r^{-p})$.
Starting with $v=o(1)$, the first assertion of
Lemma~\ref{poisson} improves the decay exponent by any strictly
smaller amount than $p-2$, until the forcing decays faster than
$r^{-3}$. This takes finitely many steps because $p>2$.
The second assertion then yields $v=O(r^{-1})$.
Substitution in \eqref{poisitr} improves the forcing to
$O(r^{-4}+r^{-p-1})$. Applying the second assertion again,
and treating the sum in the same way, gives
\begin{equation}\label{endang}
 f_i=F_i+d_i\cdot\frac{y}{r^2}+O_{C_b^1}(r^{-1-\varepsilon}),
 \qquad i=\pm,\quad 0<\varepsilon<\min\{1,p-2\}.
\end{equation}

It remains to show that one horizontal translation removes the terms
$d_+\cdot y/r^2$ and $d_-\cdot y/r^2$ simultaneously. This follows from a direct application of the balancing formula (or conservation law), see \cite{Moduli}. We use
vertical translation and with rotations mixing a horizontal
direction and the vertical direction. More precisely, let
$$e=|\nabla u|^2/2+W(u),\qquad T_{ij}=u_i u_j-e\delta_{ij}.$$
The equation implies $\partial_iT_{ij}=0$, and symmetry of $T$
gives $\diver(TX)=0$ for every $X$ obtained by differentiating Euclidean rigid motions at the identity.
The end estimates imply, uniformly on $|y|=r$,
\begin{equation}\label{moments}
 \begin{split}
 \int_{\mathbb R}u_ru_z\,dz
    &=-\sigma_0(f_{+,r}+f_{-,r})+O(r^{-3}\log r),\\
 \int_{\mathbb R}ze\,dz
    &=\sigma_0(f_++f_-)+O(r^{-2}\log r),\\
 \int_{\mathbb R}zu_ru_{y_j}\,dz
    &=O(r^{-2}\log r).
 \end{split}
\end{equation}

We now integrate the conservation law over
$\{|y|<r,\ |z|<Z\}$ and let $Z\to\infty$ with $r$ fixed.
The top and bottom fluxes vanish by the phase tails, also when
they contain a linear factor $z$.
For vertical translation, the lateral flux is
\begin{equation*}
 0=r\int_0^{2\pi}\int_{\mathbb R}u_ru_z\,dz\,d\theta.
\end{equation*}
Equations \eqref{endang} and \eqref{moments} therefore imply
\begin{equation}\label{cbal}
 c_++c_-=0,\qquad c_+=k,\quad c_-=-k,\quad k>0.
\end{equation}
For $X_j=y_j\partial_z-z\partial_{y_j}$, set $n=y/r$. We have
\begin{equation*}
 \begin{split}
 0&=r\int_0^{2\pi}\int_{\mathbb R}
       \bigl(y_ju_ru_z-zu_ru_{y_j}+ze\,n_j\bigr)\,dz\,d\theta\\
  &=\sigma_0r\sum_{i=\pm}\int_0^{2\pi}
                  n_j(f_i-rf_{i,r})\,d\theta+o(1).
 \end{split}
\end{equation*} The leading term in
$f_i-rf_{i,r}$ that depends on $\theta$ is $2d_i\cdot n/r$.
Since $\int_0^{2\pi}n_jn_l\,d\theta=\pi\delta_{jl}$,
we obtain
\begin{equation}\label{dbal}
 d_++d_-=0.
\end{equation}
Under $y=a+\widetilde y$, the coefficient $d_i$ changes to
$d_i+c_i a$. The translated radial correction has error
$O_{C_b^1}(r^{-1-(p-2)})$.
Thus $a=-d_+/k$ makes both $d_++c_+a$ and $d_-+c_-a$ equal to zero by
\eqref{cbal}--\eqref{dbal}.

\begin{proof}[Completion of the proof of Theorem~\ref{main}]
In these centered coordinates, we first define
$$
 U_{\mathrm{rad}}(r,z)
 =1+g(z-F_+(r))-g(z-F_-(r)).
$$
Proposition~\ref{todap} and \eqref{endang} then tell us
$$
 \sup_z|u-U_{\mathrm{rad}}|\le Cr^{-1-\varepsilon},\qquad
 \sup_z|\nabla_y(u-U_{\mathrm{rad}})|\le Cr^{-2-\varepsilon}.
$$
With all the decay estimates at hand, we can then use standard moving plane method to show that the solution is axially symmetric (See for instance \cite{GLW} and \cite{LY25}).  We omit the details.
\end{proof}

\section*{\bf Acknowledgements} Yong Liu was supported by  the National Key R\&D Program of
China grant 2022YFA1005400 and National Natural Science Foundation of China No. 12471204. Kelei Wang was supported by the National Key R\&D Program of China (No. 2022YFA1005602) and the National Natural Science Foundation of China (No. 12425108 and No. 12221001). Juncheng Wei was supported by National R\&D Program of China (Grant No. 2022YFA1005602), and Hong Kong General Research Fund (No. 14303125) ``On Fujita equation in the critical or supercritical regime''. Ke Wu is supported by Yunnan Revitalization Talent Support Program.
\section*{\bf Disclosures of AI assistance}
The authors used OpenAI models as assistive tools in preparing
this manuscript. The Pauli matrices, the factorization of the linearized Allen-Cahn operator and Theorem \ref{dim} were discovered by using Chat GPT 5.6.  The proof framework was developed by the authors. Mathematical arguments, proofs, and computations in the manuscript were verified and finalized by the authors.

\bibliographystyle{plain}
\bibliography{references}

\end{document}